\documentclass[11pt]{amsart}
\usepackage{amsmath,amssymb,amsfonts}
\usepackage[hidelinks]{hyperref}
\makeatletter
\def\l@section{\@tocline{1}{5pt plus 1pt}{0pt}{2.8pc}{\bfseries}}
\def\l@subsection{\@tocline{2}{1pt plus .5pt}{2.2pc}{4.8pc}{\small}}

\def\subsection{\@startsection{subsection}{2}%
  \z@{.5\linespacing\@plus.7\linespacing}{.3\linespacing}%
  {\normalfont\bfseries}}
\makeatother

\newtheorem{thm}{Theorem}[section]
\newtheorem{cor}[thm]{Corollary}
\newtheorem{lem}[thm]{Lemma}
\newtheorem{prop}[thm]{Proposition}
\theoremstyle{definition}
\newtheorem{defn}[thm]{Definition}
\theoremstyle{remark}
\newtheorem{rem}[thm]{Remark}
\numberwithin{equation}{section}

\begin{document}

\title[Structural stability of $\mathcal J_d$]{Structural stability of the Jouanolou foliations in every
degree}
\author{Sahil Gehlawat}
\address{Department of Mathematics, Indian Institute of Technology Jodhpur, NH 65, Jodhpur, Rajasthan 342030, India}
\email{sahilg@iitj.ac.in}

\begin{abstract}
We show that the degree $d$ Jouanolou foliation $\mathcal J_d$ on $\mathbb P^2$ is structurally
stable for every $d \geq 2$. 
We prove the existence of a neighbourhood $\mathcal U_d$ of
$\mathcal J_d$, such that every foliation $\mathcal F \in \mathcal U_d$ satisfies: the Fatou set is a connected disc bundle over a compact Riemann surface of genus
$d(d+1)/2$, no leaf is dense in $\mathbb P^2$, and all but countably many regular leaves are biholomorphic to the
unit disc. Finally, we prove transverse perfectness of the Julia set on the regular locus for $d \geq 2$.
\end{abstract}

\maketitle

\tableofcontents

\section{Introduction}

The global dynamics of holomorphic foliations on the complex projective plane has historically been
shaped more by phenomena of \emph{rigidity} and \emph{bifurcation} than by structural stability.
Beginning with the work of Huda\u{i}-Verenov \cite{Khu62} and Il'yashenko \cite{Ily78a,Ily78b},
classical results established strong density phenomena for generic polynomial foliations preserving
the line at infinity. Subsequent works of Cerveau \cite{Cer83}, Shcherbakov \cite{Shc82} and Nakai \cite{Nak94} clarified
the local dynamics and rigidity of nonsolvable holonomy pseudogroups. In another direction, Loray
and Rebelo \cite{LR03} constructed nonempty open sets of projective foliations in $\mathbb P^n$ with
rich dynamics, including minimality, ergodicity and rigidity phenomena. Thus a substantial part of
the classical theory suggested a picture in which complicated transverse dynamics leads naturally to
rigidity and, consequently, to bifurcation in parameter space.

It is important here to distinguish \emph{rigidity} from \emph{structural stability}. Roughly
speaking, rigidity asserts that if a nearby foliation happens to be topologically conjugate to a
given foliation, then the conjugacy is forced to come from a much more rigid analytic or projective
equivalence. Structural stability asks for a quite different conclusion: every sufficiently small
perturbation must be topologically conjugate to the original foliation. Since a projective
equivalence class occupies only a small part of the full parameter space, the rigidity mechanisms
appearing in the works above typically lead to bifurcation rather than to an open stability
component. This contrast is emphasized in the recent survey of Deroin \cite{Der26}.

We use the standard coefficient topology on the space $\operatorname{Fol}(\mathbb P^2,d)$ of
degree-$d$ foliations. Concretely, one may represent a foliation by a non-zero homogeneous
degree-$d$ vector field on $\mathbb C^3$, modulo multiplication by a non-zero scalar and addition of
a homogeneous multiple of the radial field; the quotient of the coefficient topology gives the usual
topology on $\operatorname{Fol}(\mathbb P^2,d)$. A foliation is called \emph{structurally stable} if
every sufficiently nearby foliation is carried to it by a homeomorphism of $\mathbb P^2$ which sends
singular points to singular points and regular leaves to regular leaves.

Against this background, the theorem of Alvarez and Deroin \cite{AD25} is particularly striking.
They introduced two properties $P_B$ and $P_S$, and proved that any foliation on $\mathbb P^2$
satisfying these properties is structurally stable. In particular, they proved that the degree-two
Jouanolou foliation $\mathcal J_2$ is structurally stable. More precisely, they exhibited a
nontrivial open stability component in the space of degree-two foliations and showed that the Fatou
set of $\mathcal J_2$ is a holomorphic fibration over the Klein quartic, smoothly a locally trivial
disc bundle. In particular, $\mathcal J_2$ has no dense leaf. Thus its dynamics lie in a very
different regime from the minimal and rigid dynamics that had played such a prominent role in the
preceding theory.

For $d \geq 1$, let $\mathcal J_d$ denote the Jouanolou foliation of $\mathbb P^2$ induced by the
homogeneous vector field
\begin{equation}\label{eq:Jd}
J_d = y^d\frac{\partial}{\partial x}+z^d\frac{\partial}{\partial y}+x^d\frac{\partial}{\partial z},
\end{equation}
on $\mathbb C^3$. Equivalently, this is the cyclic system
\[
\dot x = y^d,\qquad \dot y = z^d,\qquad \dot z = x^d.
\]
For $d \geq 2$, the Jouanolou foliations are classical explicit examples of projective foliations
without invariant algebraic curves; see Jouanolou \cite{Jou79}, Lins Neto \cite{Lin88}, and Brunella
\cite{Bru15}. The linear foliation $\mathcal J_1$ is exceptional in this respect and has three
invariant projective lines. The absence of algebraic leaves for $d \geq 2$ is especially significant
in view of the rigidity picture above. As stressed by Deroin \cite{Der26}, the results initiated by
Il'yashenko suggest that one should look far from foliations having invariant algebraic leaves in
order to find genuinely structurally stable projective foliations.

The Alvarez--Deroin theorem therefore raises a natural question: is the degree-two phenomenon
exceptional, or does it reflect an all-degree stability mechanism for the Jouanolou family? In
Section~4 of his ICM survey \cite{Der26}, Deroin formulates the structural-stability problem as:
\[
\mathcal J_d\quad\text{is structurally stable for every }d \geq 1.
\]
The degree-one case is verified there by the stated linear criterion and the diagonal form of
$\mathcal J_1$; see \cite[p.~110]{Der26}. Our main result proves the remaining cases.

\begin{thm}\label{thm:structural-stability}
For every integer $d \geq 2$, the Jouanolou foliation $\mathcal J_d$ is structurally stable in
$\operatorname{Fol}(\mathbb P^2,d)$.
\end{thm}

Together with the degree-one verification of Deroin,
Theorem~\ref{thm:structural-stability} gives structural stability for $\mathcal J_d$ in every
degree. For $d = 2$, the structural-stability conclusion recovers Alvarez--Deroin \cite{AD25}; the
new
metric construction is required in higher degree. For $d \geq 2$, the same construction gives a more
detailed description of the dynamics on a whole neighbourhood of the Jouanolou foliation.

\begin{thm}\label{thm:dynamics-near-jd}
For every $d \geq 2$, there exists an open neighbourhood
$\mathcal U_d\subset\operatorname{Fol}(\mathbb P^2,d)$ of $\mathcal J_d$ such that every foliation
$\mathcal F\in\mathcal U_d$ has the following properties.
\begin{enumerate}
\item The Fatou set $\mathrm F(\mathcal F)$ is a non-empty connected open set which is a smooth
locally trivial disc bundle
\[
\mathrm F(\mathcal F)\longrightarrow B_{\mathcal F},
\]
over a compact Riemann surface $B_{\mathcal F}$ of genus $d(d+1)/2$. The bundle projection is
holomorphic for the transverse complex structure.
\item No leaf of $\mathcal F$ is dense in $\mathbb P^2$.
\item All but countably many regular leaves of $\mathcal F$ are biholomorphic to the unit disc. In
particular, the Anosov conjecture holds throughout the neighbourhood $\mathcal U_d$.
\end{enumerate}
\end{thm}

The first part proves the Fatou-fibration and genus prediction formulated in
\cite[Section~4]{Der26}. Here and below, by an annulus we mean a Riemann surface conformally
equivalent to a doubly connected hyperbolic domain. For the Fatou--Julia terminology, we use
Ghys--Gomez-Mont--Saludes \cite{GGS01} and Asuke \cite{Asu10}. As already noted in \cite{AD25}, the two notions of Fatou sets are equivalent in our case. For the Julia set, we obtain one
further necessary feature of the conjectural Cantor picture. Write
\[
\mathrm J(\mathcal J_d) = \mathbb P^2\setminus\mathrm F(\mathcal J_d).
\]

\begin{thm}\label{thm:transverse-perfect}
For every integer $d \geq 2$, every regular point $p\in\mathrm J(\mathcal J_d)$, and every
sufficiently small embedded holomorphic transversal $T$ through $p$,
\[
 p\in\overline{(\mathrm J(\mathcal J_d)\cap T)\setminus\{p\}}^{\,T}.
\]
Consequently, on every open local holomorphic transversal in the regular locus, the intersection
with the Julia set is relatively closed and has no isolated points.
\end{thm}

Theorem~\ref{thm:transverse-perfect} establishes transverse perfectness only. It does not prove the
remaining parts of Deroin's conjectural Cantor picture, namely total disconnectedness and zero
transverse measure. We also prove Deroin's all-degree Jouanolou-type conjecture \cite[Conjecture~4.6]{Der26}; see
Theorem~\ref{thm:jouanolou-type}. 


\medskip
\noindent\textbf{Organization.}
In Section~2, we give a conceptual strategy of the proof and explain the role of the degree-matching
exponent, the regularization and the Alvarez--Deroin dynamical mechanism. In Section~3, we introduce
homogeneous potentials, the real field $W$, and the potential form of properties $P_B$ and $P_S$ in
our context. In Sections~4--6, we establish the $p = d+1$ Hessian estimate, repair and smooth the
potential, prove projective Levi positivity, and verify $P_S$. In Section~7, we verify the
metric-dependent parts of the Alvarez--Deroin argument for the homogeneous potential constructed
here and then apply their structural-stability mechanism. In Section~8, we apply this to
$\mathcal J_d$ for $d \geq 2$, prove Theorem~\ref{thm:dynamics-near-jd}, and prove transverse
perfectness of the
Julia set. In Section~9, we prove the all-degree Jouanolou-type conjecture. In Section~10, we record
some consequences and open parts of Deroin's conjectural picture.

\section{Strategy of the proof of Theorem \ref{thm:structural-stability}}\label{sec:strategy}

We briefly indicate the organization of the argument. The calculations are proved in full in the
later sections, while the transfer to the Alvarez--Deroin construction is reduced to its
metric-dependent inputs.

\begin{enumerate}
\item \textbf{The metric problem:} Alvarez and Deroin organize the argument around a real vector field $W$ tangent to the complex
leaves. Its zeros form an attracting transverse surface when the property $P_B$ holds, while the
property $P_S$ requires the projective singularities of the original foliation $\mathcal F$ are
hyperbolic and are sources for the real vector field $W$. Away from the sink and source neighbourhoods, longitudinal contraction
and transverse expansion lead to the invariant laminations and affine leaf geometry used in their
structural-stability proof.

For higher degree, the delicate point is the property $P_B$. The Euclidean choice used for
$\mathcal J_2$ does not give the required estimate in general; Alvarez--Deroin report numerical
evidence for other $\ell^p$ choices in \cite[Section~13.1]{AD25}. We therefore look for a
homogeneous potential adapted to the cyclic field rather than keeping the Euclidean one.

\item \textbf{The choice \texorpdfstring{$p = d+1$}{p = d+1}:} For
\[
F_p(z) = |z_1|^p+|z_2|^p+|z_3|^p
\]
and $J_d = (z_2^d,z_3^d,z_1^d)$, the three terms in the leafwise first derivative have moduli
$|z_i|^{p-1}|z_{i+1}|^d$. The adjacent variables occur with the same exponent exactly when
$p = d+1$.
Then the three moduli are
\[
(|z_1||z_2|)^d,\qquad (|z_2||z_3|)^d,\qquad (|z_3||z_1|)^d,
\]
and the critical equation becomes a triangle relation $a_1+a_2+a_3 = 0$. In Section~4, we show that the
second derivatives can be written as weighted quadratic expressions in these three sides, from which
we prove that
\[
F_{t\bar t} > |F_{tt}|,
\]
at every non-axis critical point.

\item \textbf{The coordinate directions:} The potential $F_p$ has three degenerate critical directions, namely the coordinate axes. We add
$\varepsilon S^{(d+1)/2}$ to the potential for small enough $\varepsilon > 0$, where
$S = |z_1|^2+|z_2|^2+|z_3|^2$, to repair these degeneracies. Away from fixed axis neighbourhoods,
the
strict inequality from Section~4 and compactness tell us that $P_B$ holds under a small perturbation
of the potential. We then smooth across the coordinate hyperplanes by replacing the potential with
\[
H_{\varepsilon,\eta}(z)
 = \sum_{j = 1}^3\bigl(|z_j|^2+\eta S\bigr)^{(d+1)/2}
   +\varepsilon S^{(d+1)/2}.
\]
In Section~5, we prove that, for $\varepsilon, \eta > 0$ sufficiently small, all regular critical
points of the potential $H_{\varepsilon,\eta}$ have the required Hessian sign.

\item \textbf{The normal metric and structural stability:} Set $N = H_{\varepsilon,\eta}^{1/(d+1)}$. Balanced homogeneity makes $N$ a Hermitian norm on
$\mathcal O(-1)$ and hence gives a metric on $N\mathcal J_d\simeq\mathcal O(d+2)$. Along a
Bott-parallel
normal vector, its logarithmic norm differs by a constant from a positive multiple of
$-\log H_{\varepsilon,\eta}$. Thus $P_B(H)$ is exactly the required local-maximum condition for the
infinitesimal transverse distance. In Section~6, we also prove strict projective Levi positivity and
verify that the projective singularities are sources for the associated real field.

Changing the potential changes the auxiliary real field and metrics, so one still has to check the
metric-dependent parts of the Alvarez--Deroin argument. In Section~7, we verify the sink-basin
construction, the longitudinal contraction and the transverse expansion for the potential
constructed here. Once these facts are established, the invariant foliations, affine leaf geometry,
orbital stability and band extension are exactly the arguments of \cite{AD25}. In Section~8, we
apply this to $\mathcal J_d$ for $d \geq 2$ and obtain the structural-stability theorem and its
dynamical consequences.

\end{enumerate}

\section{Homogeneous potentials and the leafwise real field}

Let $V$ be a homogeneous holomorphic vector field of degree $d \geq 1$ on $\mathbb{C}^3$, non-zero
away from the origin, and let $R = \sum_{j = 1}^3z_j\frac{\partial}{\partial z_j}$ be the radial
field.
The quotient map will be denoted by
\[
\pi:\mathbb{C}^3\setminus\{0\}\longrightarrow \mathbb{P}^2.
\]
A projective singularity is a direction for which $V(z) = \lambda z$.

\begin{defn}
A positive $C^2$ function $H$ on $\mathbb{C}^3\setminus\{0\}$ is called a balanced homogeneous potential
of degree $p > 0$, if for all $z \in \mathbb C^3 \setminus\{0\}$ and $\lambda \in \mathbb{C}^*$
\[
H(\lambda z) = |\lambda|^pH(z).
\]
We define $N = H^{1/p}$, and $f = -\log H$, the associated
logarithmic
potential.
\end{defn}

Along an integral curve $q(t) = (z_1(t), z_2(t), z_3(t))$ of the vector field $V = (V_1, V_2, V_3)$,
consider the leafwise metric
\begin{equation}\label{eq:leafmetric}
g = N^{2d-2}|dt|^2.
\end{equation}
Also, note that
\[
H_t(q(t)) = \sum_{j = 1}^{3} H_{z_j}(q(t))\frac{\partial z_j(t)}{\partial t} = \sum_{j = 1}^{3}
H_{z_j}(q(t)) V_j(q(t)).
\]
Thus we can write $H_t = \sum_{j = 1}^3 H_{z_j} V_j$. If $Y$ is a $(1,0)$ vector field, write
$\mathfrak r(Y):= Y+\overline{Y}$.

\noindent A direct computation gives the leafwise gradient of $f$.

\begin{lem}\label{lem:gradient}
The $g$-gradient of $f$ on the leaves of $V$ is
\begin{equation}\label{eq:rho}
\widetilde W = \mathfrak r(\rho V),
\end{equation}
where $\rho = -2N^{-2d+2}\frac{\overline{H_t}}{H}$. Moreover, $\widetilde W$ descends to a real
field
$W$ tangent to the projective foliation as $\rho(\lambda z) = \lambda^{1-d}\rho(z)$.
\end{lem}

\begin{proof}
Write $g = h|dt|^2$ with $h = N^{2d-2}$. In the complex leaf coordinate $t$,
\begin{align*}
\widetilde W(q(t)) &= \nabla_{g}f(q(t)) = \frac{2}{h(q(t))} \left(f_{\bar t}
\frac{\partial}{\partial t}(q(t)) + f_{t} \frac{\partial}{\partial \bar t}(\overline{q(t)})
\right)\\
&= \frac{2}{h(q(t))} \left(f_{\bar t} V(q(t)) + f_{t} \overline{V(q(t))} \right)\\
&= \mathfrak r \left( \frac{2}{h(q(t))} f_{\bar t} V(q(t)) \right).
\end{align*}
Since $f = -\log H$, therefore $f_{\bar t} = - \frac{H_{\bar t}}{H} = - \frac{\overline{H_{t}}}{H}$
and we get exactly \eqref{eq:rho}. The homogeneity of $\rho$ follows from that of $H$ and $V$.
\end{proof}

For $r > 0$, and $\theta \in \mathbb{R}$, balanced homogeneity of $H$ gives us $H(rz) = r^{p} H(z)$
and $H(e^{i\theta} z) = H(z)$ for all $z \in \mathbb C^3 \setminus \{0\}$. Differentiating and then
comparing the two equations at $r = 1$ and $\theta = 0$ gives the identity
\begin{equation}\label{eq:Euler}
\sum_{j = 1}^3H_{z_j}z_j = \frac{p}{2}H.
\end{equation}
Thus, at a projective singularity $V(z) = \lambda z$,
\begin{equation}\label{eq:sing-rho}
H_t = \frac{p}{2}\lambda H,
\qquad
\rho = -p\overline{\lambda}\,N^{-2d+2}.
\end{equation}
In particular, $H_t$ does not vanish at a projective singularity. Also recall that, at a regular
critical point, the real Hessian of $H(q(t))$ is positive definite if and only if
\begin{equation}\label{eq:PBineq}
H_{t\bar t} > |H_{tt}|.
\end{equation}
At such a point, the Hessian of $f = -\log H$ is $-H^{-1}$ times the Hessian of $H$. Therefore the
corresponding zero of $W$ is a non-degenerate sink along the leaf.

\begin{defn}\label{def:PBPS}
We say that $(V,H)$ satisfies $P_B(H)$ if every zero of $H_t$ satisfies \eqref{eq:PBineq}. We say
that $(V,H)$ satisfies $P_S(H)$ if every projective singularity is hyperbolic, in the sense that the
quotient of the two local eigenvalues is non-real, and is a source for the real field $W$. In both
cases we assume that $V$ is non-zero on $\mathbb{C}^3\setminus\{0\}$.
\end{defn}

\noindent The Levi positivity of $\log H$ will be used repeatedly below.

\begin{lem}\label{lem:metric-sign}
Assume that $H$ is smooth and that the Levi form of $\log H$ is positive on every non-radial complex
direction. Then the metric \eqref{eq:leafmetric} is complete on every regular projective leaf, its
Gaussian curvature is non-positive (and is strictly negative when $d \geq 2$), and
\begin{equation}\label{eq:divergence}
\operatorname{div}_gW = \Delta_gf = -p\Delta_g\log N < 0.
\end{equation}
\end{lem}

\begin{proof}
On the Euclidean unit sphere, $N(z)/\|z\|$ is bounded above and below by positive constants. Hence
\eqref{eq:leafmetric} is uniformly comparable with the homogeneous leaf metric used in \cite{AD25}.
Near a singular point, a local defining holomorphic field $X$ satisfies $|X(w)| \leq C|w|$, and the
induced metric dominates a constant multiple of $|dw|^2/|w|^2$. Thus a finite length path cannot
reach a singular point. Away from the singularities, completeness follows from the usual flow-box
continuation.

\noindent In complex time, $g = h|dt|^2$ with $h = N^{2d-2}$, and hence
\[
K_g = -\frac12\Delta_g\log h = -(d-1)\Delta_g\log N \leq 0.
\]
For $d \geq 2$ the inequality is strict because a projective leaf lifts to a non-radial complex
direction; for $d = 1$ the curvature is identically zero. Since $W = \nabla_g f$,
\[
\operatorname{div}_gW = \Delta_g f.
\]
Moreover $f = -\log H = -p\log N$, and therefore
\[
\operatorname{div}_gW = -\Delta_g\log H = -p\Delta_g\log N < 0,
\]
again by strict Levi positivity along the non-radial leaf direction.
\end{proof}

\begin{rem}\label{rem:gauge-dependence}
The sink and source properties $(P_B(H) \ \text{and} \ P_S(H))$ do not exhaust the metric input of
the Alvarez--Deroin \cite{AD25} argument. Their Lemma~4.2 records the leafwise curvature of the
homogeneous metric, while their Lemma~6.1 uses strict Levi positivity of the logarithmic potential
to obtain negative leafwise divergence. For a general potential, the identities above become
\[
K_g = -\frac{d-1}{p}\Delta_g\log H,
\qquad
\operatorname{div}_gW = -\Delta_g\log H.
\]
The property $P_B(H)$ controls the Hessian only at the zeros of $H_t$; it supplies no sign for
these expressions at an arbitrary regular point. Smoothness, balanced homogeneity and strict
projective Levi positivity are therefore part of the sufficient potential hypotheses used here, in
addition to $P_B(H)$ and $P_S(H)$. Strict projective Levi positivity is a convenient global
sufficient condition; the longitudinal estimate itself only needs positivity along the non-radial
leaf directions. The proof of structural stability below uses completeness and the strict divergence
estimate, not strict negativity of the Gaussian curvature, so the same criterion also covers the
linear case $d = 1$, when the critical surface is non-empty. Equivalence with the Euclidean norm
gives
metric comparison and completeness, but does not give a sign for second derivatives.

The source property is less sensitive to the potential, as was the case in \cite{AD25}. For a fixed
homogeneous representative $V$ and a projective singularity represented by $V(z) = \lambda z$,
formula~\eqref{eq:sing-rho} gives
\[
\rho_H(z) = -p\overline\lambda\,N(z)^{-2d+2}.
\]
For two smooth positive balanced potentials, the corresponding coefficients at this point differ by
a positive real factor. Since the local holomorphic defining field vanishes at the singularity, the
linearizations of the two real fields differ by that same positive factor. Thus the source character
is independent of this choice of potential. The non-real eigenvalue ratio condition in $P_S$ belongs
to the complex foliation itself. In contrast, the critical surface and $P_B$ genuinely depend on the
potential.
\end{rem}

\section{The exponent \texorpdfstring{$p=d+1$}{p=d+1} and the property \texorpdfstring{$P_B(H)$}{PB(H)}}

We begin with a general power potential in order to explain the choice of exponent and then turn
directly to the property $P_B(H)$ for the Jouanolou field. For $p \ge 2$ and $d \ge 2$, let
\[
V_d = (z_2^d,z_3^d,z_1^d),
\qquad
H_p(z) = \sum_{i = 1}^3|z_i|^p,
\]
with cyclic indices. This is the same Jouanolou field as \eqref{eq:Jd}.

\subsection{Why \texorpdfstring{$p = d+1$}{p = d+1}?}\label{subsec:why-p}

Let $Z(t)$ be a holomorphic integral curve of $V_d$, so
\begin{equation}\label{eq:curve-first}
 z_i' = z_{i+1}^d,
 \qquad
 z_i'' = d\,z_{i+1}^{d-1}z_{i+2}^d.
\end{equation}
For the one-variable function $\phi_p(w) = |w|^p = (w\bar w)^{p/2}$, Wirtinger differentiation
gives,
for $w\neq0$,
\begin{equation}\label{eq:wirtinger-power}
 (\phi_p)_w = \frac p2|w|^{p-2}\bar w,
 \qquad
 (\phi_p)_{w\bar w} = \frac{p^2}{4}|w|^{p-2},
 \qquad
 (\phi_p)_{ww} = \frac{p(p-2)}4|w|^{p-4}\bar w^2.
\end{equation}
Consequently
\begin{equation}\label{eq:Fp-first}
 (H_p)_t = \frac p2\sum_{i = 1}^3 a_i(p),
 \qquad
 a_i(p):=|z_i|^{p-2}\bar z_i\,z_{i+1}^d.
\end{equation}
At a regular critical point of $H_p$, we have
\begin{equation}\label{eq:critical-closure-general}
 a_1(p)+a_2(p)+a_3(p) = 0.
\end{equation}
The moduli of these three complex numbers are
\begin{equation}\label{eq:general-side-moduli}
 |a_i(p)| = |z_i|^{p-1}|z_{i+1}|^d.
\end{equation}
We want the side length attached to the ordered pair $(z_i,z_{i+1})$ to be symmetric in the two
adjacent coordinates. This requires $p = d+1$,
and with this choice,
\begin{equation}\label{eq:symmetric-side-moduli}
 |a_i| = (|z_i||z_{i+1}|)^d.
\end{equation}
Thus the critical equation \eqref{eq:critical-closure-general} says that the three complex numbers
$a_i$ close up as the sides of a possibly degenerate Euclidean triangle. We shall see below that the
same choice of exponent also puts the second derivative terms into a form adapted to this triangle
relation.

\noindent From now on, we put $p = d+1$ and
\[
F_{d}:= H_{d+1} = \sum_{i = 1}^3|z_i|^{d+1}.
\]

\subsection{Reduction of \texorpdfstring{$P_B(H)$}{PB(H)} for the Jouanolou foliation}

We first write the condition of Section~3 explicitly for $F_d$. For $z_i\neq0$, set
\[
R_i = |z_i|^{d+1},
\qquad
a_i = |z_i|^{d-1}\overline{z_i}\,z_{i+1}^d.
\]
Since $Z'(t) = V_d(Z(t))$, formula \eqref{eq:wirtinger-power} gives
\begin{align}
(F_d)_t
&=\sum_i (\phi_p)_{z_i}z_i'
 = \frac {d+1}{2}\sum_i|z_i|^{d-1}\bar z_i z_{i+1}^d
 = \frac {d+1}{2}\sum_i a_i,
\label{eq:Ft}\\
(F_d)_{t\bar t}
&=\sum_i(\phi_p)_{z_i\bar z_i}|z_i'|^2
 = \frac{(d+1)^2}{4}\sum_i|z_i|^{d-1}|z_{i+1}|^{2d}
 = \frac{(d+1)^2}{4}\sum_i\frac{|a_i|^2}{R_i}.
\label{eq:Ftbart}
\end{align}
For the pure $t$-derivative there are two contributions. The chain rule gives
\begin{equation}\label{eq:hessian-acceleration-schematic}
\frac{d^2}{dt^2}F_d(Z(t))
 = (F_d)_{zz}(Z(t))[Z'(t),Z'(t)]
+(F_d)_{z}(Z(t))[Z''(t)].
\end{equation}
Substituting \eqref{eq:wirtinger-power} and \eqref{eq:curve-first},
\begin{align*}
(F_d)_{tt}
 = {}&\frac{(d+1)(d-1)}4\sum_i|z_i|^{d-3}\bar z_i^{\,2}z_{i+1}^{2d}
+\frac{(d+1)d}{2}\sum_i|z_i|^{d-1}\bar z_i z_{i+1}^{d-1}z_{i+2}^d\\
 = {}&\frac{d+1}{2}\left(
\frac{d-1}{2}\sum_i\frac{a_i^2}{R_i}
+d\sum_i\frac{a_{i-1}a_i}{R_i}
\right).
\end{align*}
Thus, we have
\begin{equation}\label{eq:Ftt}
(F_d)_{tt} = \frac {d+1}{2}\left(
\frac{d-1}{2}\sum_i\frac{a_i^2}{R_i}
+d\sum_i\frac{a_{i-1}a_i}{R_i}
\right).
\end{equation}

\noindent At a critical point, \eqref{eq:Ft} becomes
\begin{equation}\label{eq:triangle-closure-special}
a_1+a_2+a_3 = 0.
\end{equation}
Put $t_i:=|a_i| = (|z_i||z_{i+1}|)^d$ and $r:=\frac {d+1}{d} = 1+\frac{1}{d}\in(1,2)$. The factors
$R_i^{-1}$ can now be expressed in terms of the side lengths. Indeed,
\[
t_{i+1}^r
 = (|z_{i+1}||z_{i+2}|)^{dr}
 = (|z_{i+1}||z_{i+2}|)^{d+1}
 = R_{i+1}R_{i+2},
\]
and therefore
\begin{equation}\label{eq:weightidentity}
\frac1{R_i} = \frac{t_{i+1}^r}{R_1R_2R_3}.
\end{equation}
Set $w_i = t_{i+1}^r$, and write
\[
U = \sum_iw_i a_i^2,
\qquad
C = \sum_iw_i a_i a_{i-1},
\qquad
Q = \sum_iw_i|a_i|^2.
\]
Substitution of \eqref{eq:weightidentity} into \eqref{eq:Ftbart} and \eqref{eq:Ftt} gives
\[
(F_d)_{t\bar t}
 = \frac{(d+1)^2}{4R_1R_2R_3}Q,
\quad
(F_d)_{tt} = \frac {d+1}{4 R_1 R_2 R_3}\left(
(d-1)U
+ 2dC
\right).
\]
Since $(d-1)U+2dC = d(U+2C)-U$. Hence
\begin{equation}\label{eq:ratio}
\frac{|(F_d)_{tt}|}{(F_d)_{t\bar t}}
 = \frac{|d(U+2C)-U|}{(d+1)Q}.
\end{equation}
The elementary estimate $|U| \leq Q$ leaves one term to control. In view of
\eqref{eq:triangle-closure-special}, the needed estimate is a weighted inequality for the sides of a
triangle, namely a bound of $|U+2C|$ by $Q$ for the particular weights $w_i = t_{i+1}^r$. We prove
this next.

\begin{rem}\label{rem:AD-gauges}
For $d = 2$, Alvarez--Deroin work with the Euclidean potential, corresponding to $p = 2$. The
coordinate
critical points are then non-degenerate, but the global verification of $P_B$ at the critical points
is difficult; their proof reduces the required estimate to a finite collection of exact integer
inequalities which is checked by computer \cite[Section~12.3]{AD25}. For higher degrees, they also
consider $\ell^p$-type norms and report numerical evidence, in particular, for the choice $p = d$ in
degrees $3,4,5$ \cite[Section~13.1]{AD25}. The calculation above explains why we use $p = d+1$
instead. For this exponent the critical equation gives the symmetric side lengths
\eqref{eq:symmetric-side-moduli}, and the Hessian condition reduces to the triangle inequality
below.
\end{rem}

\subsection{The \texorpdfstring{$P_B(H)$}{PB(H)} estimate}

We now prove the estimate required by \eqref{eq:ratio}. More generally, let
$a_1,a_2,a_3\in\mathbb C$ satisfy
\[
a_1+a_2+a_3 = 0,
\]
and, for non-negative weights $w = (w_1,w_2,w_3) \in \mathbb R_{\ge 0}^3$, put
\[
U(w) = \sum_iw_ia_i^2,
\qquad
C(w) = \sum_iw_ia_ia_{i-1},
\ \ \text{and} \qquad
Q(w) = \sum_iw_i|a_i|^2.
\]
Let
\[
\mathcal K = \{w\in \mathbb R_{\ge 0}^3: |U(w)+2C(w)| \leq Q(w)\}.
\]
Since $U, C$ and $Q$ depend linearly on the weights, $\mathcal K$ is a closed convex cone.

\begin{lem}\label{lem:clipped}
For every real number $h \geq 0$, if $w_i(h) = \min\{|a_{i+1}|^2,h\}$, for all $1 \le i \le 3$, then the weight $w(h) = (w_1(h), w_2(h), w_3(h)) \in \mathcal K$
\end{lem}

\begin{proof}
We first record the three elementary cases that occur as $h$ varies. If $c \ge 0$, then
$w = (c,c,c) \in \mathcal K$ as
\[
U(w)+2C(w) = c(a_1+a_2+a_3)^2 = 0 \le Q(w).
\]
Next, the closure relation gives
$a_i^2+2a_ia_{i-1} = (a_i + a_{i-1})^2 - a_{i-1}^2 = a_{i+1}^2-a_{i-1}^2,$
which gives
\begin{equation}\label{eq:U2C-weight-difference}
U(w)+2C(w) = \sum_i(w_{i-1}-w_{i+1})a_i^2.
\end{equation}
We now prove that any non-negative weight vector whose maximum is attained at least twice belongs to
$\mathcal K$. Indeed, if for instance $w_1 = w_3 = M \geq w_2 = m$, then
\[
|U+2C| = (M-m)|a_1^2-a_3^2|
 \leq (M-m)(|a_1|^2+|a_3|^2)
 \leq Q,
\]
and the other two cases follow by cyclic permutation. Finally, if $w_i = |a_{i+1}|^2$ for
$1 \leq i \leq 3$, then
\[
C = \sum_i|a_{i+1}|^2a_ia_{i-1}
 = a_1a_2a_3\sum_j\overline{a_j} = 0,
\]
and therefore $|U+2C| = |U| \leq Q$.

We use the above three cases to prove that $w(h) \in \mathcal K$. Order the three numbers $|a_i|^2$
as $m \leq M \leq L$. The cyclic shift in the definition of $w_i(h)$ is irrelevant.
\begin{itemize}
    \item If $0 \leq h \leq m$, then $w(h) = (h,h,h)$ and therefore $w(h) \in \mathcal K$.
    \item If $m \leq h \leq M$, at least two weights are equal to the common maximum $h$, so the
    preceding two-maxima case applies.
    \item If $M \leq h \leq L$, the weight vector $w(h)$ lies on the line segment joining its values
    at $h_0 = M$ and $h_1 = L$. The first weight vector $w(h_0) \in \mathcal K$ by the two-maxima
    case,
    while the second weight vector $w(h_1) = (|a_2|^2, |a_3|^2, |a_1|^2)$ is the squared-side vector
    just considered above. Convexity of $\mathcal K$ gives $w(h) \in \mathcal K$.
    \item If $h \geq L$, then $w(h) = (|a_2|^2, |a_3|^2, |a_1|^2)$ and therefore
    $w(h) \in \mathcal K$.
\end{itemize}

\end{proof}

\begin{lem}\label{lem:concave}
Assume $a_i\neq0$, put $t_i = |a_i|$, and let $0 < r < 2$. For the weights $w_i = t_{i+1}^r$ one has
\begin{equation}\label{eq:defect}
|U+2C|
 \leq Q-\left(1-\frac r2\right)(\min_i t_i)^r\sum_it_i^2
 < Q.
\end{equation}
\end{lem}

\begin{proof}
Put $\alpha = r/2\in(0,1)$. We use the following integral identity
\begin{equation}\label{eq:layer-cake}
x^\alpha = \alpha(1-\alpha)\int_0^\infty\min\{x,h\}h^{\alpha-2}\,dh,
\qquad x > 0.
\end{equation}
Indeed, splitting the integral at $h = x$ gives
\[
\int_0^\infty\min\{x,h\}h^{\alpha-2}\,dh
 = \int_0^x h^{\alpha-1}\,dh+x\int_x^\infty h^{\alpha-2}\,dh
 = \frac{x^\alpha}{\alpha(1-\alpha)}.
\]
Apply \eqref{eq:layer-cake} with $x = t_{i+1}^2$. If $w_i(h):=\min\{t_{i+1}^2,h\}$, then
\begin{equation}\label{eq:weights-layer-cake}
w_i = t_{i+1}^r = (t_{i+1}^2)^{\alpha}
 = \alpha(1-\alpha)\int_0^\infty w_i(h)h^{\alpha-2}\,dh.
\end{equation}
Because $U$, $C$ and $Q$ are linear in the weights,
\begin{align}
U(w)+2C(w)
&=\alpha(1-\alpha)\int_0^\infty
\bigl(U(w(h))+2C(w(h))\bigr)h^{\alpha-2}\,dh,
\label{eq:A-layer}\\
Q(w)
&=\alpha(1-\alpha)\int_0^\infty Q(w(h))h^{\alpha-2}\,dh.
\label{eq:Q-layer}
\end{align}
Lemma~\ref{lem:clipped} gives for all $h \ge 0$,
\begin{equation}\label{eq:clipped-ineq-expanded}
|U(w(h))+2C(w(h))| \leq Q(w(h)).
\end{equation}
This gives $|U(w)+2C(w)| \leq Q(w)$ after integration. To obtain the strict inequality, we put
$m = \min_i t_i > 0$. For $0 < h < m^2$, all three weights $w_i(h)$ are equal to $h$, and hence
\[
U(w(h))+2C(w(h)) = h(a_1+a_2+a_3)^2 = 0,
\quad \text{and} \quad
Q(w(h)) = h\sum_it_i^2 > 0.
\]
Thus, we can start the integral in \eqref{eq:A-layer} from $m^2$, whereas the corresponding positive
part of \eqref{eq:Q-layer} remains. Using \eqref{eq:clipped-ineq-expanded},
\begin{align*}
|U(w)+2C(w)|
& \leq \alpha(1-\alpha)\int_{m^2}^\infty Q(w(h))h^{\alpha-2}\,dh\\
&=Q-\alpha(1-\alpha)\int_0^{m^2}h\,h^{\alpha-2}\,dh\sum_it_i^2\\
&=Q-(1-\alpha)m^{2\alpha}\sum_it_i^2 < Q(w).
\end{align*}
Since $\alpha = r/2$, this is \eqref{eq:defect} and thus
$w = (|a_2|^r, |a_3|^r, |a_1|^r) \in \mathcal K$.
\end{proof}

We now return to the weights coming from the Jouanolou critical point. Lemma~\ref{lem:concave},
together with $|U| \leq Q$, gives from \eqref{eq:ratio}
\begin{align*}
\frac{|(F_d)_{tt}|}{(F_d)_{t\bar t}} =& \frac{|d(U+2C)-U|}{(d+1)Q} \leq \frac{d}{d+1}
\frac{|U+2C|}{Q} + \frac{1}{d+1}\\
 \leq & \frac{d}{d+1}\left(1 - \left(1-\frac{r}{2}\right) \frac{(\min_i t_i)^r \sum_i t_i^2}{Q}
\right) + \frac{1}{d+1} \\
 \leq & 1-\frac d{d+1}\left(1-\frac r2\right)
\frac{(\min_i t_i)^r\sum_i t_i^2}{Q}.
\end{align*}
Thus we have
\begin{equation}\label{eq:ratio-gap}
    \frac{|(F_d)_{tt}|}{(F_d)_{t\bar t}} \le 1-\frac d{d+1}\left(1-\frac r2\right)
\frac{(\min_i t_i)^r\sum_i t_i^2}{Q}.
\end{equation}
The right-hand side is strictly smaller than $1$ whenever all three $t_i$ are non-zero. Hence, the
$\ell^{d+1}$ potential has positive-definite leafwise real Hessian at every non-axis critical point.

It remains to identify what happens when a coordinate vanishes. If exactly one coordinate is zero
and the other two are non-zero, then \eqref{eq:Ft} contains exactly one non-zero summand, so the
point is not critical. Thus a critical point on a coordinate hyperplane lies on a coordinate axis.
At $e_1 = (1,0,0)$, the vector field $V_d$ satisfies $V_d(e_1) = (0,0,1)$. Since $p > 2$, the second
real
derivatives of $|w|^p$ vanish at $w = 0$, and direct substitution gives
\[
(F_d)_t(e_1) = 0,
\qquad
(F_d)_{t\bar t}(e_1) = 0,
\quad \text{and} \ \
(F_d)_{tt}(e_1) = 0.
\]
The other two axes are identical by cyclic symmetry. Thus, the only degenerate critical directions
of the potential $F_d$ are the three coordinate directions.

For the perturbation in the next section, we shall only need a uniform margin on sets which stay
away from these directions. This will follow immediately from compactness.

\begin{cor}\label{cor:hessian-gap}
Fix $d \geq 2$. Let
$K \subset \partial \mathbb B = \{(z_1, z_2, z_3) \in \mathbb C^3 : S(z) = |z_1|^2 + |z_2|^2 + |z_3|^2 = 1\}$
be compact and disjoint from the three coordinate axes. Then
there is a number $\gamma_{d,K} > 0$ such that, at every leafwise critical point $q\in K$ of $F_d$,
\begin{equation}\label{eq:compact-hessian-gap}
(F_d)_{t\bar t}(q)-|(F_d)_{tt}(q)| \geq \gamma_{d,K}.
\end{equation}
\end{cor}

\begin{proof}
The critical set
\[
\Sigma_K = \{q\in K:(F_d)_t(q) = 0\}
\]
is compact. If it is empty there is nothing to prove. Otherwise every point of $\Sigma_K$ is a
non-axis critical point, and hence \eqref{eq:ratio-gap} gives
\[
(F_d)_{t\bar t}-|(F_d)_{tt}| > 0
\qquad\text{on }\Sigma_K.
\]
The left-hand side is continuous on $K$, so its minimum on $\Sigma_K$ is positive. Taking this
minimum for $\gamma_{d,K}$ proves the claim.
\end{proof}

\begin{rem}[The coordinate degeneracies]\label{rem:coordinate-degeneracies}
The choice $p = d+1$ has a feature which is already visible in degree two. For $d = 2$, the
Euclidean
choice $p = 2$ used in \cite{AD25} does not have the coordinate degeneracies above, whereas our
choice
$p = 3$ does. On the other hand, for $p = d+1$ the Hessian estimate away from the axes follows
analytically from the triangle argument, and Corollary~\ref{cor:hessian-gap} supplies the uniform
margin needed on every fixed compact set avoiding them. In Section~5, we modify the potential near
the three coordinate directions and then smooth it, while preserving the $P_B$ inequality. The same
argument applies for every $d \geq 2$.
\end{rem}

\section{Removing the coordinate degeneracies}

Put $S = |z_1|^2+|z_2|^2+|z_3|^2$, and for $\varepsilon > 0$ define
\[
F_{d,\varepsilon}:=F_d+\varepsilon S^{\frac{d+1}{2}}.
\]
The perturbation will lift the three zero Hessians of the potential $F_d$. The axis neighbourhoods
below are chosen once and for all, depending only on $d$, and are independent of $\varepsilon$.

\begin{lem}\label{lem:axis}
In the projective chart $z_1 = 1$, write $y = z_2$, $z = z_3$, and put $\rho_d = (6d)^{-1/d}$. For
$|y|,|z| < \rho_d$ and $(y,z)\neq(0,0)$ one has
\[
(F_d)_{t\bar t} > |(F_d)_{tt}|,
\]
while both sides vanish when $y = z = 0$. Moreover, for $G = S^{\frac{d+1}{2}}$,
\[
G_{t\bar t} > |G_{tt}|
\]
throughout the neighbourhood $|y|,|z| < \rho_d$. Consequently, every leafwise critical point of
$F_{d,\varepsilon} = F_d+\varepsilon G$ lying in this fixed neighbourhood has positive-definite real
Hessian, for every $\varepsilon > 0$. The same statement holds near the other two coordinate points.
\end{lem}

\begin{proof}
Let $s = \frac{d+1}{2}$ and
\[
L = |y|^{2d}+|y|^{d-1}|z|^{2d}+|z|^{d-1}.
\]
Formula \eqref{eq:Ftbart} gives directly
\begin{equation}\label{eq:axis-Ftbart}
(F_d)_{t\bar t} = s^2L.
\end{equation}
We now separate the two contributions to $(F_d)_{tt}$ described in
\eqref{eq:hessian-acceleration-schematic}. In the chart $z_1 = 1$, the first term is
\[
P = s(s-1)\left(
 y^{2d}+|y|^{d-3}\bar y^{\,2}z^{2d}+|z|^{d-3}\bar z^{\,2}
\right).
\]
Here and below, terms with a negative exponent are understood by their continuous extensions when
the corresponding factor vanishes.
Taking moduli term by term gives
\begin{equation}\label{eq:axis-ambient-bound}
|P| \leq s(s-1)L = \frac{(d+1)(d-1)}4L.
\end{equation}
The second term is
\[
A = sd\left(
 y^{d-1}z^d+|y|^{d-1}\bar y\,z^{d-1}+|z|^{d-1}\bar z\,y^d
\right),
\]
and therefore
\begin{equation}\label{eq:axis-acc-bound-start}
|A| \leq sd\bigl(
|y|^{d-1}|z|^d+|y|^d|z|^{d-1}+|y|^d|z|^d
\bigr).
\end{equation}
If $|y|,|z| < \rho_d < 1$, each of the three products in the above inequality is at most
$\rho_d^d|z|^{d-1} \leq \rho_d^dL$. Since $\rho_d^d = (6d)^{-1}$, we get
\begin{equation}\label{eq:axis-acc-bound}
|A| \leq 3sd\rho_d^dL = \frac{s}{2}L = \frac{(d+1)}{4} L.
\end{equation}
Combining \eqref{eq:axis-Ftbart}, \eqref{eq:axis-ambient-bound} and \eqref{eq:axis-acc-bound},
\[
|(F_d)_{tt}|
 \leq \left(s(s-1)+\frac s2\right)L
 = s\left(s-\frac12\right)L
 < s^2L = (F_d)_{t\bar t}
\]
whenever $(y,z)\neq(0,0)$. At $(y,z) = (0,0)$, both derivatives vanish, as noted at the end of
Section~4.

\noindent We next treat the Euclidean homogeneous correction. Along the Jouanolou field,
\begin{align}
S_t&=\sum_i\bar z_i z_{i+1}^d,\label{eq:S-t}\\
S_{t\bar t}&=\sum_i|z_i|^{2d},\label{eq:S-tbart}\\
S_{tt}&=d\sum_i\bar z_i z_{i+1}^{d-1}z_{i+2}^d.\label{eq:S-tt}
\end{align}
In the chart $z_1 = 1$, \eqref{eq:S-tbart} gives $S_{t\bar t} \geq 1$. Moreover
\[
S_{tt} = d\bigl(y^{d-1}z^d+\bar y z^{d-1}+\bar z y^d\bigr),
\]
so $|S_{tt}| \leq 3d\rho_d^d = \frac{1}{2}$. Let $G = S^s$. The chain rule gives
\begin{align}
G_{t\bar t}
&=sS^{s-1}S_{t\bar t}+s(s-1)S^{s-2}|S_t|^2,
\label{eq:G-tbart}\\
G_{tt}
&=sS^{s-1}S_{tt}+s(s-1)S^{s-2}S_t^2.
\label{eq:G-tt}
\end{align}
Since $|S_t^2| = |S_t|^2$, we have
\begin{align}
G_{t\bar t}-|G_{tt}|
& \geq sS^{s-1}\bigl(S_{t\bar t}-|S_{tt}|\bigr)\\
& \geq \frac{s}{2}S^{s-1} > 0.
\label{eq:G-axis-gap}
\end{align}
At the axis point itself, $S = 1$, $S_t = S_{tt} = 0$ and $S_{t\bar t} = 1$, so $G_{tt} = 0$ and
$G_{t\bar t} = s > 0$. Finally, for every $\varepsilon > 0$,
\begin{align*}
&(F_{d,\varepsilon})_{t\bar t}-|(F_{d,\varepsilon})_{tt}|\\
&\qquad \geq \bigl((F_d)_{t\bar t}-|(F_d)_{tt}|\bigr)
+\varepsilon\bigl(G_{t\bar t}-|G_{tt}|\bigr) > 0
\end{align*}
throughout the fixed neighbourhood. At a critical point, the inequality is exactly the
positive-definiteness criterion for the real Hessian. Thus, any critical point which moves into, or
is newly created inside the neighbourhood under the perturbation, automatically has the required
Hessian sign.
\end{proof}

\begin{prop}\label{prop:eps}
For $d \geq 2$, and $\varepsilon > 0$ small enough, the potential $F_{d,\varepsilon}$ has
positive-definite real Hessian at every leafwise critical point.
\end{prop}

\begin{proof}
It is enough to work on the unit sphere
$\partial \mathbb B = \{(z_1, z_2, z_3) \in \mathbb C^3 : S(z) = 1\}$ after normalizing by $S$.
Let $U$ be the union of the three fixed axis neighbourhoods
from Lemma~\ref{lem:axis}, and put
\[
K = \partial \mathbb{B}\setminus U.
\]
The above lemma already gives the required Hessian sign throughout $U$ for every $\varepsilon > 0$. Next,
consider the critical set of the unperturbed potential in $K$,
\[
\Sigma_K = \{q\in K:(F_d)_t(q) = 0\}.
\]
If $\Sigma_K$ is empty, then $|(F_d)_t|$ has a positive minimum on $K$, and for sufficiently small
$\varepsilon$ the perturbed potential has no critical point there. Suppose therefore that $\Sigma_K$
is non-empty. By Corollary~\ref{cor:hessian-gap}, there is $\gamma > 0$ such that on $\Sigma_K$, we
have
\[
(F_d)_{t\bar t}-|(F_d)_{tt}| \geq \gamma.
\]
Choose a neighbourhood $N$ of $\Sigma_K$ in $K$ on which the same gap is larger than $\gamma/2$. On
the compact set $K\setminus N$, the function $|(F_d)_t|$ has a positive minimum, say $c > 0$.

\noindent Since $F_{d,\varepsilon} = F_d+\varepsilon S^{p/2}$ converges to $F_d$ in $C^2(A)$ as
$\varepsilon\downarrow0$, for all sufficiently small $\varepsilon$ we have
\[
|(F_{d,\varepsilon})_t| > c/2
\qquad\text{on }K\setminus N,
\]
while on $N$,
\[
(F_{d,\varepsilon})_{t\bar t}-|(F_{d,\varepsilon})_{tt}| > \gamma/4.
\]
Thus no critical point of $F_{d,\varepsilon}$ occurs in $K\setminus N$, and every critical point in
$N$ has positive-definite real Hessian. Together with Lemma~\ref{lem:axis}, this proves the
proposition.
\end{proof}

If $d$ is odd, $F_{d,\varepsilon}$ is already smooth. Let $\eta > 0$. For all $d \geq 2$, and in
particular uniformly across even and odd $d$, we use the explicit regularization
\begin{equation}\label{eq:smoothing}
F_{d,\varepsilon,\eta}(z)
 = \sum_{j = 1}^3\bigl(|z_j|^2+\eta S\bigr)^{\frac{d+1}{2}}+\varepsilon S^{\frac{d+1}{2}}.
\end{equation}
For every $\eta > 0$, $F_{d,\varepsilon, \eta}$ is $C^\infty$ on $\mathbb C^3\setminus\{0\}$ and
balanced homogeneous of degree $d+1$. Notice that the roles of the two parameters are different:
$\varepsilon$ repairs the axis degeneracies, while $\eta$ provides smoothness across the coordinate
hyperplanes.

\begin{lem}\label{lem:C2}
For $d \ge 2$ and fixed $\varepsilon > 0$,
\[
F_{d,\varepsilon,\eta}\longrightarrow F_{d,\varepsilon}
\]
in $C^2$ on every compact annulus $\{z \in \mathbb C^3 : 0 < c \leq S(z) \leq C\}$ as
$\eta\downarrow0$.
\end{lem}

\begin{proof}
Write $s = \frac{d+1}{2} \geq \frac{3}{2}$. On a compact annulus $\tilde A = \{c \leq S \leq C\}$,
$S$
and
its first two real derivatives are uniformly bounded. For one summand put
\[
r = |z_j|^2,
\qquad
q_\eta = r+\eta S.
\]
Then $q_\eta\to r$ in $C^\infty(\tilde A)$ and
\begin{align}
D(q_\eta^s)&=s q_\eta^{s-1}Dq_\eta,\label{eq:C2-first}\\
D^2(q_\eta^s)&=s q_\eta^{s-1}D^2q_\eta
+s(s-1)q_\eta^{s-2}Dq_\eta\otimes Dq_\eta.\label{eq:C2-second}
\end{align}
The zeroth and first derivatives converge uniformly because $s > 1$, $q_\eta-r = O(\eta)$, and
$Dr = O(r^{1/2})$.

For the second derivatives, the first term in \eqref{eq:C2-second} converges uniformly since
$q_\eta^{s-1}\to r^{s-1}$ uniformly and $D^2q_\eta = D^2r+O(\eta)$. If $s \geq 2$ the last term is
also
immediate, because the power $q_\eta^{s-2}$ is non-singular. The only endpoint requiring an estimate
is $s = 3/2$, which occurs when $d = 2$. In that case
\[
q_\eta^{-1/2}Dq_\eta\otimes Dq_\eta-r^{-1/2}Dr\otimes Dr
\]
is decomposed by adding and subtracting $q_\eta^{-1/2}Dr\otimes Dr$. Since $|Dr| \leq C r^{1/2}$ and
$q_\eta \geq r+c\eta$ on $A$, the coefficient-difference term is bounded by
\[
Cr\left|q_\eta^{-1/2}-r^{-1/2}\right| \leq C\eta^{1/2}.
\]
The cross terms are bounded by
\[
C\eta q_\eta^{-1/2}r^{1/2} \leq C\eta,
\]
and the $\eta^2DS\otimes DS$ term by $C\eta^{3/2}$. Hence the last term in \eqref{eq:C2-second} also
converges uniformly. Thus each summand converges in $C^2(\tilde A)$, and summing over $j$ proves the
lemma.
\end{proof}

\begin{prop}\label{prop:eta}
Fix $d \geq 2$ and choose $\varepsilon > 0$ as in Proposition~\ref{prop:eps}. Then, for all
sufficiently
small $\eta > 0$, every leafwise critical point of $F_{d,\varepsilon,\eta}$ satisfies
\[
(F_{d,\varepsilon,\eta})_{t\bar t} > |(F_{d,\varepsilon,\eta})_{tt}|.
\]
\end{prop}

\begin{proof}
Again normalize by $S = 1$. Put $G = F_{d,\varepsilon}$. Its projective critical set
\[
\Sigma_G = \{z\in \partial \mathbb B:G_t(z) = 0\}
\]
is compact. Proposition~\ref{prop:eps} implies that the continuous gap function
$\Gamma_G = G_{t\bar t}-|G_{tt}|$ is strictly positive on $\Sigma_G$. Hence
\[
\gamma:=\min_{\Sigma_G}\Gamma_G > 0.
\]
Choose an open neighbourhood $U$ of $\Sigma_G$ on which $\Gamma_G > \gamma/2$. On the compact
complement $\partial \mathbb B\setminus U$, the continuous function $|G_t|$ has a positive minimum,
say $c > 0$.

\noindent By Lemma~\ref{lem:C2}, $F_{d,\varepsilon,\eta}\to G$ in $C^2(\partial\mathbb B)$. For
sufficiently small $\eta$, we therefore have $|(F_{d,\varepsilon,\eta})_t| > c/2$ on
$\partial \mathbb B\setminus U$, so no new critical point can appear there, while on $U$ the
corresponding Hessian gap
is larger than $\gamma/4$. Thus every critical point of the smoothed potential lies in $U$ and has
the required strict Hessian sign.
\end{proof}

\section{Levi positivity and the property \texorpdfstring{$P_S(H)$}{PS(H)}}

\subsection{The Levi form}

Set $H:=F_{d,\varepsilon,\eta}$, $s = \frac{d+1}{2}$ and
\[
Q_j = |z_j|^2+\eta S.
\]
The Levi calculation is independent of the critical set and of the Jouanolou field. It is a
pointwise statement about the regularized homogeneous potential $H$.

We first record the elementary calculation for one positive-definite Hermitian quadratic form. To
avoid any convention about which variable of the Hermitian product is complex linear, write
\[
Q(z) = z^*Az,
\]
where $A$ is a positive-definite Hermitian matrix. Then $\partial Q_z(\xi) = z^*A\xi$, and the Levi
form
\[
\mathcal L_Q(z;\xi) = \xi^*A\xi.
\]
Therefore
\begin{equation}\label{eq:levi-logQ}
\mathcal L_{\log Q}(z;\xi)
 = \frac{(z^*Az)(\xi^*A\xi)-|z^*A\xi|^2}{(z^*Az)^2}.
\end{equation}
By Cauchy--Schwarz for the Hermitian form defined by $A$,
\begin{equation}\label{eq:levi-logQ-positive}
\mathcal L_{\log Q}(z;\xi) \geq 0,
\end{equation}
with equality exactly when $\xi\in\mathbb Cz$. For our forms,
\[
Q_j(z) = |z_j|^2+\eta S(z)
\]
has diagonal matrix with one entry $1+\eta$ and the other two entries $\eta$, and is therefore
positive definite for every $\eta > 0$. The Euclidean form $S$ is also positive definite. Thus every
$\log Q_j$ and $\log S$ has non-negative Levi form with radial kernel exactly $\mathbb Cz$.

\medskip
\noindent We now pass from the individual quadratic forms to their sum of powers. Write
\[
H = \sum_{j = 1}^3e^{u_j}+e^{u_4},
\]
where $u_j = s\log Q_j$, for $1 \leq j \leq 3$ and $u_4 = \log\varepsilon+s\log S$. Define
\[
w_k:=\frac{e^{u_k}}{H},
\]
and note that $w_k > 0$ and $\sum_kw_k = 1$. Since $\partial H = \sum_ke^{u_k}\partial u_k$, we have
\begin{equation}\label{eq:dlogH}
\partial\log H = \sum_kw_k\partial u_k.
\end{equation}
The derivative of a weight $w_k$ is
\[
\partial w_k = \frac{e^{u_k}}{H} \partial u_k - \frac{e^{u_k}}{H^2} \partial H
 = w_k\left(\partial u_k-\sum_jw_j\partial u_j\right).
\]
Applying $\bar\partial$ to \eqref{eq:dlogH}, evaluating on $(\xi,\bar\xi)$, and collecting the terms
gives the log-sum-exp identity
\begin{equation}\label{eq:levi-logsum}
\mathcal L_{\log H}(\xi)
 = \sum_kw_k\mathcal L_{u_k}(\xi)
+\sum_kw_k|\partial u_k(\xi)|^2
-\left|\sum_kw_k\partial u_k(\xi)\right|^2.
\end{equation}
For completeness, let $c_k = \partial u_k(\xi)$. The last two terms can be written as the weighted
variance
\begin{equation}\label{eq:weighted-variance}
\sum_kw_k|c_k|^2-
\left|\sum_kw_kc_k\right|^2
 = \frac12\sum_{j,k}w_jw_k|c_j-c_k|^2 \geq 0.
\end{equation}
Furthermore, $\mathcal L_{u_j} = s\mathcal L_{\log Q_j}$, for $1 \leq j \leq 3$ and
$\mathcal L_{u_4} = s\mathcal L_{\log S}$. If $\xi\notin\mathbb Cz$, each term in the first sum of
\eqref{eq:levi-logsum} is strictly positive by \eqref{eq:levi-logQ-positive}, and all $w_k$ are
positive. Hence, for all $\xi\notin\mathbb Cz$
\begin{equation}\label{eq:strict-projective-Levi}
\mathcal L_{\log H}(z;\xi) > 0.
\end{equation}

\noindent It remains to check the radial direction. If $\xi = \lambda z$, homogeneity of every
quadratic form gives
\[
\mathcal L_{\log Q_j}(z;\lambda z) = 0,
\quad \text{and} \ \
\mathcal L_{\log S}(z;\lambda z) = 0.
\]
Moreover, $\partial\log Q_j(z)(\lambda z) = \lambda \frac{Q_j(z)}{Q_j(z)} = \lambda$, and
$\partial\log S(z)(\lambda z) = \lambda \frac{S(z)}{S(z)} = \lambda$, so
\[
\partial u_k(\lambda z) = s\lambda
\]
for every $k$. The variance \eqref{eq:weighted-variance} therefore also vanishes. Consequently
\begin{equation}\label{eq:Levi-kernel-radial}
\ker\mathcal L_{\log H}(z;\cdot) = \mathbb Cz.
\end{equation}
This radial kernel is forced by balanced homogeneity as
\[
\log H(\lambda z) = \log H(z)+(d+1)\log|\lambda|,
\]
and $\log|\lambda|$ is pluriharmonic on $\mathbb C^*$.

\subsection{The induced normal metric}\label{subsec:normal-metric-construction}

Put $N = H^{1/p}$. Balanced homogeneity gives
\[
N(\lambda z) = |\lambda|N(z).
\]
Thus $N^2$ defines a smooth Hermitian metric on the tautological line bundle
$\mathcal O(-1)\to\mathbb P^2$. A degree-$d$ foliation of $\mathbb P^2$ has normal bundle
\[
N\mathcal F\simeq\mathcal O(d+2),
\]
so duality and tensor powers induce a Hermitian normal metric whose local frame has norm
proportional to $N^{-(d+2)}$. For any homogeneous vector field $V$ of degree $d$ and zero divergence,
including $V = J_d$, it is useful to see this directly. Let
$R = \sum z_j\frac{\partial}{\partial z_j}$,
$\Omega = dz_1\wedge dz_2\wedge dz_3$, and define the normal $1$-form
\[
\vartheta_z(u):=\det(R(z),V(z),u) = \bigl(\iota_V\iota_R\Omega\bigr)_z(u).
\]
It annihilates the radial and leaf directions and satisfies
\[
\vartheta_{\lambda z}(\lambda u) = \lambda^{d+2}\vartheta_z(u),
\]
which is the transition law for $\mathcal O(d+2)$. Since $\operatorname{div}V = 0$ and
$[V,R] = (1-d)V$,
the Lie derivative satisfies
\[
\mathrm L_V\vartheta = 0.
\]
To see this, we use the identity $\mathrm L_{X}(\iota_{Y}) - \iota_Y(\mathrm L_X) = \iota_{[X,Y]}$,
for vector fields $X,Y$. Therefore,
\begin{align*}
\mathrm L_V\vartheta =& \iota_{[V,V]}(\iota_R \Omega) + \iota_V(\mathrm L_V(\iota_R \Omega))\\
=& \iota_V\iota_{[V,R]}(\Omega) + \iota_V\iota_R (\mathrm L_V \Omega)\\
=& \iota_V\iota_{(1-d)V}(\Omega) + \iota_V\iota_R (\operatorname{div}V \Omega)\\
=& 0.
\end{align*}

\noindent Hence this trivialization is Bott-parallel along the lifted leaves; compare
\cite[Lemma~6.2]{AD25}. The real field $W$ is a leafwise time change of $V$. Differentiating the
time-change factor contributes only a vector tangent to the foliation, which disappears in the
normal quotient. Thus the map induced by $D\Phi_W^t$ on the normal quotient is Bott parallel
transport. If $\xi(t)$ is a non-zero Bott-parallel normal vector, its coefficient in the
$\vartheta$-trivialization is constant, while the induced frame norm is $N(z(t))^{-(d+2)}$.
Therefore
\begin{equation}\label{eq:inf-distance-core}
\log\|\xi(t)\|_\perp
 = C-(d+2)\log N(z(t))
 = C+\frac{d+2}{p}f(z(t)),
\qquad f = -\log H.
\end{equation}
At a regular critical point $H_t = 0$,
\[
f_{t\bar t} = -\frac{H_{t\bar t}}H,
\quad \text{and} \ \
f_{tt} = -\frac{H_{tt}}H.
\]
Thus $H_{t\bar t} > |H_{tt}|$ is equivalent to negative definiteness of the real Hessian of $f$, and
by \eqref{eq:inf-distance-core} to a strict local maximum of the infinitesimal transverse distance.

\begin{thm}\label{thm:gauge}
Fix $d \geq 2$ and put $p = d+1$. There exist $\varepsilon,\eta > 0$ such that
\begin{equation}\label{eq:Hepseta}
H_{\varepsilon,\eta}(z)
 = \sum_{j = 1}^3\bigl(|z_j|^2+\eta S\bigr)^{\frac{d+1}{2}}
+\varepsilon S^{\frac{d+1}{2}}
\end{equation}
is positive, $C^\infty$ on $\mathbb C^3\setminus\{0\}$, and balanced homogeneous of degree $d+1$. It
has the following properties:
\begin{enumerate}
\item[(i)] every zero of the leafwise derivative $H_t$ along $J_d$ satisfies
\[
H_{t\bar t} > |H_{tt}|;
\]
\item[(ii)] the Levi form of $\log H$ is non-negative and its kernel at $z\neq0$ is exactly the
radial line $\mathbb Cz$;
\item[(iii)] $N = H^{\frac{1}{d+1}}$ induces a smooth Hermitian metric on
$N\mathcal J_d\simeq\mathcal O(d+2)$ for which every regular critical point of the infinitesimal
transverse distance is a strict
local maximum.
\end{enumerate}
\end{thm}

\begin{proof}
Choose $\varepsilon > 0$ as in Proposition~\ref{prop:eps}, and then choose $\eta > 0$ sufficiently
small
as in Proposition~\ref{prop:eta}. This gives (i), together with positivity, smoothness and balanced
homogeneity. Equations \eqref{eq:levi-logsum}--\eqref{eq:Levi-kernel-radial} give (ii). The
construction and identity \eqref{eq:inf-distance-core} give (iii).
\end{proof}

\begin{rem}[Role of the two regularization parameters]\label{rem:eps-eta-roles}
The smallness requirements on $\varepsilon$ and $\eta$ enter only in verifying the $P_B$ property. The Levi
calculation itself needs only $\varepsilon > 0$ and $\eta > 0$: for every such pair, the forms
$Q_j = |z_j|^2+\eta S$ are positive definite and we get strict Levi positivity on projective
directions. Thus $\varepsilon$ first repairs the axis Hessians, $\eta$ then smooths the power
potential without losing $P_B$, while projective Levi positivity is automatic for the resulting
positive regularized family.
\end{rem}

\subsection{The property \texorpdfstring{$P_S(H)$}{PS(H)}}

For the adapted potential, the projective singularities can be computed explicitly. In the affine
chart $x = 1$, with coordinates $(y,z)$, a holomorphic defining field is
\begin{equation}\label{eq:localfield}
X = (z^d-y^{d+1})\frac{\partial}{\partial y}
+(1-zy^d)\frac{\partial}{\partial z}.
\end{equation}
Its singular points are
\begin{equation}\label{eq:singularities}
[1:a:a^{-d}],
\quad \text{where} \ \
a^{d^2+d+1} = 1.
\end{equation}
At $[1:1:1]$, the derivative of \eqref{eq:localfield} is
\[
\begin{pmatrix}
-(d+1)&d\\
-d&-1
\end{pmatrix},
\]
whose eigenvalues are
\begin{equation}\label{eq:eigenvalues}
\mu_\pm = -\frac{d+2}{2}\pm i\frac{\sqrt3\,d}{2}.
\end{equation}
Their quotient is non-real. Since $J_d(1,1,1) = (1,1,1)$, formula \eqref{eq:sing-rho} gives a
negative
real coefficient for the associated real field, and hence $[1:1:1]$ is a source. For
$a^{d^2+d+1} = 1$, put $D_a = \operatorname{diag}(1,a,a^{-d})$. Then
\[
J_d(D_a z) = a^dD_aJ_d(z),\qquad H(D_a z) = H(z).
\]
The second identity follows directly from the modulus-only formula \eqref{eq:Hepseta}. Since
$|a| = 1$, the first identity and \eqref{eq:rho} imply $\widetilde W(D_a z) = D_a\widetilde W(z)$.
Thus
these symmetries preserve the real field as well as the foliation, and carry the source $[1:1:1]$ to
every singularity in \eqref{eq:singularities}.

\noindent The preceding calculation proves the following proposition.

\begin{prop}\label{prop:PS}
For every $d \geq 2$, the Jouanolou field $J_d$ equipped with the potential of
Theorem~\ref{thm:gauge}
satisfies $P_S(H)$.
\end{prop}

\noindent This is consistent with the direct higher-degree computation in \cite[Section~13.1]{AD25}.

\section{The Alvarez--Deroin mechanism for a homogeneous potential}

Alvarez and Deroin \cite[Proposition~2.2 and Theorem~10.1]{AD25} prove that a degree-$d$ algebraic
foliation of $\mathbb P^2$ satisfying their properties $P_B$ and $P_S$ is structurally stable. Their
proof is written for the real field and auxiliary metrics associated with the standard Hermitian
norm. For the potential constructed above, the critical surface, the real field, the leafwise metric
and the induced normal metric are different. We therefore verify the metric-dependent parts of their
argument. Strict projective Levi positivity is needed in addition to the sink and source properties.
Once the sink-basin construction, longitudinal contraction and transverse expansion have been
recovered, the remaining arguments are those of Alvarez--Deroin.

\begin{thm}\label{thm:gauge-criterion}
Let $V$ be a homogeneous vector field of degree $d \geq 1$ on $\mathbb{C}^3$, non-zero away from the
origin, and of zero divergence. Let $\mathcal{F}$ be the induced foliation of $\mathbb{P}^2$.
Suppose that the critical set $B = \pi(\{H_t = 0\})$ is non-empty and that a smooth positive
balanced
homogeneous potential $H$ satisfies:
\begin{enumerate}
\item[(i)] $\log H$ has positive Levi form on every non-radial direction;
\item[(ii)] $(V,H)$ satisfies $P_B(H)$;
\item[(iii)] $(V,H)$ satisfies $P_S(H)$.
\end{enumerate}
Then $\mathcal{F}$ is structurally stable.
\end{thm}

The dependence on the three hypotheses is the following. Condition~(ii) supplies the sink-basin
construction of \cite[Theorem~5.1]{AD25} and the growth estimate outside the sink neighbourhood;
condition~(i) gives the longitudinal contraction used in \cite[Lemma~6.1]{AD25}; the induced normal
metric, together with the growth estimate, gives the transverse expansion of
\cite[Lemma~6.2 and Corollary~6.3]{AD25}; and condition~(iii) gives the source product models of
\cite[Lemma~4.11]{AD25}. These are precisely the inputs used from \cite[Section~8]{AD25} onward.

\begin{proof}
From this point on, we use the notation of \cite[Section~3]{AD25} as much as possible. Put
\[
S = \operatorname{Sing}(\mathcal F),\qquad
B = \pi(\{H_t = 0\}),
\quad \text{and} \quad
\mathbb P^2_* = \mathbb P^2\setminus(B\cup S).
\]
By \eqref{eq:sing-rho}, $B$ does not meet the singular set. At a point of $B$, the condition
$H_{t\bar t} > |H_{tt}|$ says that the real differential of $H_t$ in the leaf direction is
invertible.
Hence $B$ is a smooth real surface transverse to the complex foliation. For $d \geq 2$, its
non-emptiness follows from $T\mathcal F\simeq\mathcal O(1-d)$ exactly as in \cite[Corollary~4.4]{AD25};
in Theorem~\ref{thm:gauge-criterion} it is included as a hypothesis so that the statement also
covers $d = 1$.

The property $P_B(H)$ gives a small closed plaque neighbourhood $U_B$ of $B$ whose boundary is
inward pointing for $W$. Along the lifted flow, with $f = -\log H$, we have
\begin{equation}\label{eq:qgrowth}
q = df(\widetilde W) = |\widetilde W|_g^2
 = 4N^{-2d+2}\frac{|H_t|^2}{H^2}.
\end{equation}
The function $q$ is scalar invariant, vanishes exactly over $B$, and is strictly positive at the
projective singularities. Since $\mathbb P^2\setminus\operatorname{Int}(U_B)$ is compact, there are
constants $0 < a \leq A$ such that
\begin{equation}\label{eq:fgrowth}
at \leq f(\widetilde\Phi_t(z))-f(z) \leq At
\end{equation}
for every positive orbit segment which stays outside $U_B$. The lifted real flow is complete in both
time directions, since \eqref{eq:rho} makes $\widetilde W$ a real homogeneous field of degree one
and therefore $|\widetilde W(z)| \leq C|z|$.

The proof of \cite[Theorem~5.1]{AD25} now applies with $f = -\log H$ in place of the Euclidean
logarithmic potential. In particular,
\begin{equation}\label{eq:disc-basin}
D = \operatorname{Sat}_{\mathcal F}(B) = \operatorname{Att}_W(B)
\end{equation}
is a smooth locally trivial open-disc bundle over $B$; every leaf in $D$ meets $B$ at exactly one
point, and the projection
\[
P:D\longrightarrow B
\]
is holomorphic for the transverse complex structure. Here
\[
\operatorname{Att}_W(B) = \{x:\ d(\Phi_W^t(x),B)\longrightarrow0\text{ as }t\to+\infty\}.
\]
Similarly, following \cite{AD25}, we write
\[
\operatorname{Rep}_W(S)
 = \{x:\ d(\Phi_W^t(x),S)\longrightarrow0\text{ as }t\to-\infty\}
 = \operatorname{Att}_{-W}(S).
\]

\noindent We next check the two estimates used in \cite[Sections~6--8]{AD25}. On $\mathbb P^2_*$, let
\[
E_\ell = T_{\mathbb R}\mathcal F/\mathbb RW.
\]
If $\omega_g$ is the leafwise area form and $\|[v]\|_\ell = |\omega_g(v,W)|$, then the logarithmic
derivative of this norm is $\operatorname{div}_gW$, as in \cite[Lemma~6.1]{AD25}. By
\eqref{eq:divergence},
\[
\operatorname{div}_gW = -\Delta_g\log H < 0.
\]
Thus the quotient $T_{\mathbb R}\mathcal F/\mathbb RW$ is uniformly exponentially contracted on
compact sets away from $B$ and $S$.

For the direction transverse to the foliation, the normal $1$-form and the Hermitian normal metric
were constructed in Subsection~\ref{subsec:normal-metric-construction}. The computation there is
independent of the particular potential, and the induced map of $D\Phi_W^t$ on
$T\mathbb P^2/T\mathcal F$ is Bott parallel transport; compare \cite[Lemma~6.2 and
Corollary~6.3]{AD25}. Hence
\eqref{eq:inf-distance-core} gives, along a lifted orbit,
\begin{equation}\label{eq:normalgrowth}
\frac{\|v(t)\|_\perp}{\|v(0)\|_\perp}
 = \left(\frac{N(z(0))}{N(z(t))}\right)^{d+2}
 = \exp\left(\frac{d+2}{p}[f(z(t))-f(z(0))]\right).
\end{equation}
Together with \eqref{eq:fgrowth}, this gives uniform transverse expansion outside the sink
neighbourhood.

By $P_S(H)$, every singularity is a hyperbolic source for $W$. The source neighbourhoods and their
boundary foliations are therefore those of \cite[Lemma~4.11]{AD25}. Combining these source and sink
neighbourhoods with the preceding two estimates gives the complete metric and the uniform estimates
of \cite[Lemma~8.1]{AD25}. In particular, if
\[
K = \mathbb P^2\setminus\bigl(\operatorname{Att}_W(B)\cup\operatorname{Rep}_W(S)\bigr),
\]
then the proof of \cite[Proposition~8.3]{AD25} gives the invariant splitting, and
\cite[Proposition~8.8]{AD25} gives the invariant foliations denoted there by
\[
\mathcal F_W^{--},\qquad \mathcal F_W^{-},\qquad
\mathcal F_W^{++},\qquad \mathcal F_W^{+}.
\]
The affine structure is then obtained as in \cite[Section~9]{AD25}; in particular, every non-simply-connected leaf is a topological annulus. The proof of structural stability in \cite[Theorem~10.1 and
Sections~10.1--10.3]{AD25} uses these sink and source product structures, the invariant foliations
and the affine leaf geometry; no further property of the Euclidean norm is used.

It remains to verify that the preceding structures persist under sufficiently small degree-$d$
perturbations. Fix the potential $H$; its strict projective Levi positivity is independent of the
foliation. If homogeneous representatives $V_n\to V$ in coefficients, then $H_t$, $H_{tt}$ and
$H_{t\bar t}$ converge uniformly on the unit sphere. Let
\[
\Sigma = \{z\in\partial\mathbb B:H_t^V(z) = 0\},
\qquad
\Gamma = H_{t\bar t}^V-|H_{tt}^V|.
\]
The set $\Sigma$ is compact and $\Gamma$ has a positive minimum on it. Choose a neighbourhood $U$ of
$\Sigma$ in the unit sphere on which $\Gamma$ remains positive. On the compact complement of $U$,
$|H_t^V|$ has a positive minimum. Uniform $C^2$ convergence therefore prevents new critical points
outside $U$ and preserves the positive Hessian gap inside $U$. Moreover, the invertibility of the
real leafwise differential of $H_t$ along $B$ and the implicit-function theorem show that the
non-empty transverse critical surface $B$ persists for all sufficiently small perturbations. Hence
$P_B(H)$ is open. Hyperbolicity of the projective singularities, the source property $P_S(H)$, and
non-vanishing of the homogeneous representative on $\mathbb C^3\setminus\{0\}$ are open as well.

The zero divergence hypothesis does not restrict the nearby projective foliations. If $\widehat V$
is homogeneous of degree $d$, set
\begin{equation}\label{eq:df}
\widehat V^{\mathrm{df}} = \widehat V-\frac{\operatorname{div}\widehat V}{d+2}R.
\end{equation}
If $P$ is homogeneous of degree $d-1$, then
\[
\operatorname{div}(PR) = (d+2)P.
\]
Thus $\widehat V^{\mathrm{df}}$ is divergence-free and induces the same projective foliation; this
is the continuous linear slice used in \cite[Lemma~4.1]{AD25}. After shrinking the neighbourhood,
the constructions and the conclusions of \cite[Sections~5--11]{AD25} used above therefore remain
valid for every nearby foliation. This proves the theorem.
\end{proof}

We shall also use two consequences of the same argument. First, the proof of
\cite[Proposition~7.1]{AD25}, with \eqref{eq:disc-basin} and \eqref{eq:normalgrowth} in place of the
corresponding Euclidean estimates, gives
\begin{equation}\label{eq:FatouD}
\mathrm F(\mathcal F) = D = \operatorname{Att}_W(B).
\end{equation}
Second, the proof of \cite[Proposition~11.5]{AD25} applies without change once its two
metric-dependent inputs are replaced by \eqref{eq:disc-basin} and the transverse expansion
\eqref{eq:normalgrowth}. Its characteristic-class and transverse-torus arguments do not use the
Euclidean norm. Consequently, for $d \ge 2$, $B$ is connected, has degree $1-d$, and
\begin{equation}\label{eq:genusB}
g(B) = \frac{d(d+1)}2.
\end{equation}

\section{Application to the Jouanolou family}

As mentioned earlier, the structural stability of $\mathcal J_1$ is verified by Deroin in \cite{Der26} using the
degree-one linear criterion. We therefore only have to consider $d \geq 2$.

\begin{proof}[Proof of Theorem~\ref{thm:structural-stability}]
For $J_d = (y^d,z^d,x^d)$, the homogeneous vector field has zero divergence and is non-zero away
from
the origin. Choose $\varepsilon > 0$ and then $\eta > 0$ as in Sections~4--5, and let
$H = H_{\varepsilon,\eta}$, with $p = d+1$. Sections~4--6 verify $P_B(H)$, strict projective Levi
positivity and $P_S(H)$. Moreover, the formula for $H$ gives $H_t(e_1) = 0$, so the critical set $B$
is non-empty. Theorem~\ref{thm:gauge-criterion} therefore gives structural stability in
$\operatorname{Fol}(\mathbb P^2,d)$. For $d = 2$, this recovers the conclusion of \cite{AD25}, while
for $d \geq 3$ it gives the higher-degree conclusion.
\end{proof}

\begin{proof}[Proof of Theorem~\ref{thm:dynamics-near-jd}]
After shrinking the neighbourhood given by Theorem~\ref{thm:structural-stability}, Section~7 and the
corresponding results of Alvarez--Deroin apply to every foliation in it. Equations~\eqref{eq:FatouD}
and \eqref{eq:genusB} give part~(1). A leaf outside the Fatou domain misses the non-empty open Fatou
set and therefore cannot be dense, while a leaf inside it is one fibre of the disc bundle and misses
every sufficiently small open set lying over a different point of the base. This proves part~(2).

By \cite[Theorem~9.5]{AD25}, all but countably many regular leaves are simply connected, while the remaining leaves are annuli with hyperbolic holonomy. Lins Neto
\cite{Lin00} proved hyperbolicity when the singularities are non-degenerate and the tangent line
bundle admits a metric of negative curvature. Here $T\mathcal F\simeq\mathcal O(1-d)$, which is
negative for $d \geq 2$. The singularities of $\mathcal J_d$ are non-degenerate, and this property
persists under sufficiently small perturbations. Hence every simply connected regular leaf in a
sufficiently small neighbourhood is biholomorphic to $\mathbb D$. This proves part~(3).
\end{proof}

\subsection{Transverse perfectness of the Julia set}

We now prove Theorem~\ref{thm:transverse-perfect}. The argument uses only the properties established
in Section~7: the identification \eqref{eq:FatouD} of the Fatou set with the sink basin, the
holomorphic transverse projection in \eqref{eq:disc-basin}, the normal expansion from
\eqref{eq:fgrowth} and \eqref{eq:normalgrowth}, and the genus computation \eqref{eq:genusB}. As in the proof of Alvarez--Deroin \cite[Proposition~11.5]{AD25}, normal expansion forces the derivative of the basin projection to blow up when the evaluation point approaches the boundary of its domain. Here this is quantified using the first-entry time into the sink neighbourhood $U_{B}$ and combined with removability of punctures for maps into the compact hyperbolic curve $B$.


\begin{lem}\label{lem:puncture-removable}
Let $\Sigma$ be a compact Riemann surface of genus at least two. Every holomorphic map
\[
g:\mathbb{D}^* = \mathbb{D}\setminus\{0\}\longrightarrow\Sigma
\]
extends holomorphically across $0$.
\end{lem}

\begin{proof}
Let $q:\mathbb{D}\to\Sigma$ be the universal covering and let
\[
E:\mathbb H\longrightarrow\mathbb{D}^*,\qquad E(\zeta) = \exp(2\pi i\zeta),
\]
be the universal covering of the punctured disc. Lift $g\circ E$ to a holomorphic map
$G:\mathbb H\to\mathbb{D}$. There is a deck transformation $\gamma$ of $q$ such that
\[
G(\zeta+1) = \gamma G(\zeta).
\]
By Schwarz--Pick,
\[
d_\mathbb{D}\bigl(G(iy),\gamma G(iy)\bigr)
 = d_\mathbb{D}\bigl(G(iy),G(iy+1)\bigr)
 \leq d_{\mathbb H}(iy,iy+1)\longrightarrow0
\]
as $y\to+\infty$. The deck group of the universal cover of a compact hyperbolic surface is
torsion-free and cocompact. Hence every non-identity deck transformation has strictly positive
translation length. It follows that $\gamma$ is the identity.

Thus $G$ is $1$-periodic and descends to a bounded holomorphic map $h:\mathbb{D}^*\to\mathbb{D}$ with
$g = q\circ h$. By the removable-singularity theorem, $h$ extends holomorphically to $\mathbb{D}$. Its
value at $0$ lies in $\mathbb{D}$: if the extended value had modulus one, the maximum principle would
force the extension to be constant of modulus one, contrary to $h(\mathbb{D}^*)\subset\mathbb{D}$. Therefore
$q\circ h$ gives the required holomorphic extension of $g$.
\end{proof}

\begin{proof}[Proof of Theorem~\ref{thm:transverse-perfect}]
Fix $d \geq 2$ and let
\[
D = \mathrm F(\mathcal J_d),\quad \text{and} \ \
J = \mathrm J(\mathcal J_d) = \mathbb P^2\setminus D.
\]
By \eqref{eq:FatouD} and \eqref{eq:disc-basin}, $D$ is exactly the forward basin of the compact
transverse sink surface $B$, and the basin projection $\Pi:D\longrightarrow B$ is holomorphic for the transverse complex structures. By \eqref{eq:genusB}, $B$ is connected and
\[
g(B) = \frac{d(d+1)}2 \geq 3.
\]
Choose the closed sink plaque bundle $U_B\Subset D$ used in Section~7, with boundary transverse to
$W$.

Suppose, if possible, that a regular point $p\in J$ is isolated in $J$ on a sufficiently
small embedded holomorphic transversal. After shrinking the transversal, choose a holomorphic
parametrization
\[
i:\Delta\longrightarrow\mathbb P^2,
\]
such that $i(0) = p$, and is transverse to the foliation on $\Delta$, is disjoint from $U_B$, and satisfies $i(\Delta^*)\subset D$. Define
\[
F = \Pi\circ i:\Delta^*\longrightarrow B.
\]

For $z\neq0$, let $\tau(z) > 0$ be the first forward time at which $\Phi_W^{\tau(z)}(i(z))$ reaches
$\partial U_B$. This time is finite because $i(z)\in D$ and $D$ is the forward basin of $B$. We
claim that
\[
\lim_{z \to 0}{\tau(z)} = +\infty.
\]
Otherwise, there would be a sequence $z_n\to0$ for which $\tau(z_n)$ is bounded. Passing to a
subsequence, let $\tau(z_n)\to t_* \geq 0$. Since $\partial U_B$ is closed and the flow is
continuous,
\[
\Phi_W^{t_*}(p)\in\partial U_B\subset D.
\]
The basin $D$ is invariant under every finite flow time, so this would imply $p\in D$, contradicting
$p\in J$.

We next estimate the transverse derivative of $F$. On the regular locus let $N\mathcal J_d$ denote
the complex normal line of the foliation, equipped with the Hermitian norm used in
\eqref{eq:normalgrowth}. Since $\Pi$ is a transverse holomorphic submersion, its differential
induces a complex-linear isomorphism
\[
\overline{D\Pi}_q:N\mathcal J_{d,q}\longrightarrow T_{\Pi(q)}B,
\]
where $q\in D$. Fix a smooth Hermitian metric on $B$. Compactness of $\partial U_B$ gives $m_d > 0$ such that the
conorm of $\overline{D\Pi}_q$ is at least $m_d$ for every $q\in\partial U_B$. Since $i$ is a fixed
transverse holomorphic disc, after shrinking $\Delta$ there is $a_d > 0$ such that
\[
\|[Di_z(v)]\|_\perp \geq a_d|v|
\]
for every $z\in\Delta$ and $v\in T_z\Delta$.

For the orbit segment from $i(z)$ to its first entry point in $U_B$, the orbit stays outside
$\operatorname{Int}U_B$. Hence \eqref{eq:fgrowth} and \eqref{eq:normalgrowth} give a constant
$c_d > 0$ such that
\[
\bigl\|[D\Phi_W^{\tau(z)}Di_z(v)]\bigr\|_\perp
 \geq e^{c_d\tau(z)}\|[Di_z(v)]\|_\perp.
\]
Because $\Pi$ is constant along the leaves,
\[
\Pi\circ\Phi_W^t = \Pi
\]
whenever both sides are defined in $D$. For each fixed $z\neq0$, we differentiate this identity at
the fixed time $t = \tau(z)$; no derivative of $\tau$ is involved. Combining the preceding estimates
gives
\begin{equation}\label{eq:Julia-derivative-growth}
\|DF_z\| \geq m_da_d e^{c_d\tau(z)}\longrightarrow+\infty
\quad \text{as} \ \ z\to0.
\end{equation}
Equivalently, if one differentiates the variable-time entry map, the additional term is proportional
to $W$ and disappears in the normal quotient and under $D\Pi$.

On the other hand, $B$ is a compact Riemann surface of genus at least two.
Lemma~\ref{lem:puncture-removable} therefore extends $F$ holomorphically across $0$. The derivative
of the extended map is bounded on a smaller closed disc with respect to the fixed metrics,
contradicting \eqref{eq:Julia-derivative-growth}. Hence $p$ cannot be isolated in the Julia set on a
local holomorphic transversal.

Finally, $J$ is closed because $D$ is open. Thus its intersection with an open local transversal is
relatively closed, and the first part shows that it has no isolated points. This proves the theorem.
\end{proof}

\begin{rem}\label{rem:perfectness-perturbations}
The proof uses only the disc-basin Fatou fibration, hyperbolicity of its compact base, and the
normal-expansion estimate up to first entry. These properties persist for the sufficiently small
degree $d$ perturbations covered by Theorem~\ref{thm:dynamics-near-jd} and Section~7. After
shrinking the perturbation neighbourhood if necessary, the same argument therefore gives transverse
perfectness at every regular Julia point for those nearby foliations as well.
\end{rem}

\section{The Jouanolou-type property}

Deroin asserts that $\mathcal J_d$ is of Jouanolou type for every $d \geq 1$
\cite[Conjecture~4.6]{Der26}. We prove this all-degree conclusion here. 

\begin{defn}\label{def:jouanolou-type}
A holomorphic foliation $\mathcal F$ on a compact complex algebraic surface is said to be of
\emph{Jouanolou type} if:
\begin{enumerate}
\item every singularity of $\mathcal F$ is hyperbolic;
\item there exists a Hermitian metric on the normal bundle for which the infinitesimal transverse
distance, restricted along the leaves, has local maxima at its critical points;
\item no leaf of $\mathcal F$ is biholomorphic to an elliptic curve.
\end{enumerate}
\end{defn}

\begin{thm}\label{thm:jouanolou-type}
For every integer $d \geq 1$, the Jouanolou foliation $\mathcal J_d$ is of Jouanolou type.
\end{thm}

\subsection{The linear case \texorpdfstring{$d = 1$}{d = 1}}

The endpoint $d = 1$ does not require the higher-degree curvature argument. In fact, the
distinguished exponent $p = d+1$ is then $2$, and the adapted potential collapses to a positive
constant multiple of the Euclidean quadratic potential. We record a direct verification because the
linear endpoint behaves somewhat differently from the higher-degree case.

\begin{prop}\label{prop:d1-jouanolou}
The linear Jouanolou foliation $\mathcal J_1$ is of Jouanolou type.
\end{prop}

\begin{proof}
Consider the linear map $P : \mathbb C^3 \to \mathbb C^3$ given by $P(z_1,z_2,z_3) = (z_2,z_3,z_1)$.
The degree one Jouanolou foliation is given by the vector field $J_1(z) = P(z)$. Here, we have
\[
H_1(z) = S(z) = |z_1|^2+|z_2|^2+|z_3|^2,
\quad \text{and} \ \ N_1 = H_1^{1/2}.
\]
Equivalently, the regularized expression used for $d \geq 2$ becomes
\[
H_{\varepsilon,\eta} = (1+3\eta+\varepsilon)S
\]
when $d = 1$, so no coordinate-direction repair is needed.

\noindent We first check the metric condition. Use the standard Hermitian product
$\langle u,v\rangle = \sum_ju_j\overline{v_j}$. Along a complex integral curve $z'(t) = P(z(t))$,
\[
(H_1)_t = \langle Pz,z\rangle,
\qquad
(H_1)_{t\bar t} = \langle Pz,Pz\rangle = H_1,
\quad \text{and} \ \
(H_1)_{tt} = \langle P^2z,z\rangle.
\]
Since $P$ is unitary and $P^* = P^2$,
\[
(H_1)_{tt} = \langle z,Pz\rangle
 = \overline{\langle Pz,z\rangle}
 = \overline{(H_1)_t}.
\]
Hence at every regular critical point of $H_1$ along a leaf, that is, $(H_1)_t = 0$ gives
\[
(H_1)_{tt} = 0,
\quad \text{and} \ \
(H_1)_{t\bar t} = H_1 > 0.
\]
Thus the Hessian inequality is strict with the maximal possible gap. Since $f_1 = -\log H_1$, these
critical points are strict local maxima of $f_1$. The norm $N_1$ gives the standard Hermitian metric
on $\mathcal O(-1)$ and therefore a Hermitian metric on
$N\mathcal J_1\simeq\mathcal O(3)$. The Bott-homogeneity computation underlying
\eqref{eq:inf-distance-core} is algebraic and remains valid at $d = 1$; for a non-zero Bott-parallel
normal vector $\xi(t)$ it gives
\[
\log\|\xi(t)\|_\perp
 = C-3\log N_1(z(t))
 = C+\frac32 f_1(z(t)).
\]
Consequently the critical points of the infinitesimal transverse distance are precisely the above
critical points and are strict local maxima. This is condition~(2) of
Definition~\ref{def:jouanolou-type}.

\noindent For condition~(1), the projective singularities are the three eigenlines of $P$. At
$[1:1:1]$, the local field is
\[
X = (z-y^2)\frac{\partial}{\partial y}+(1-zy)\frac{\partial}{\partial z},
\]
whose linearization is
\[
DX(1,1) = \begin{pmatrix}
-2&1\\
-1&-1
\end{pmatrix}.
\]
Its eigenvalues are $-\frac32\pm i\frac{\sqrt3}{2}$, so their quotient is non-real. For each cube
root of unity $a$, the diagonal map
\[
D_a = \operatorname{diag}(1,a,a^{-1})
\]
satisfies $P(D_a z) = aD_aP(z)$ and sends $[1:1:1]$ to one of the three eigenlines of $P$. Hence
these
projective symmetries carry the above local model to the other two singularities, which are
therefore hyperbolic as well. We also record the source property needed for structural stability.
At $[1:1:1]$ the homogeneous relation is $Pz = z$, so \eqref{eq:sing-rho} with $p = 2$ and $d = 1$
gives
$\rho = -2$. The linearization of the real field is the real part of $-2X_{\text{lin}}$, where $X_{\text{lin}}$ is the linear part of $X$ at $[1:1:1]$; its two
complex
eigenvalues have real part $3 > 0$. Thus $[1:1:1]$ is a source for $W$. The source property is
preserved by these symmetries as well. Indeed,
\[
(H_1)_t(D_a z) = a\,(H_1)_t(z),
\qquad
P(D_a z) = aD_aP(z),
\]
and $|a| = 1$. Formula~\eqref{eq:rho} therefore gives
\[
\rho(D_a z)\,P(D_a z) = D_a\bigl(\rho(z)P(z)\bigr).
\]
Thus, $D_a$ conjugates the real field $W$, not only the projective complex foliation. The other two
singularities are consequently sources as well. Hence $(J_1,H_1)$ satisfies $P_S(H_1)$.

\noindent It remains to show that it excludes elliptic leaves. Let $\omega = e^{2\pi i/3}$. In an
eigenbasis $(u,v,w)$ for $P$,
\[
P = \operatorname{diag}(1,\omega,\omega^2).
\]
On the affine chart $u\neq0$, with $x = v/u$ and $y = w/u$, the induced foliation is generated by
\[
\dot x = (\omega-1)x,
\qquad
\dot y = (\omega^2-1)y.
\]
If $x_0y_0\neq0$, the leaf through $(x_0,y_0)$ is parametrized by
\[
t\longmapsto
\bigl(x_0e^{(\omega-1)t},\,
      y_0e^{(\omega^2-1)t}\bigr).
\]
A non-zero period would give integers $m,n$ with $(\omega-1)t = 2\pi i m$, and
$(\omega^2-1)t = 2\pi i n$. For $t \neq 0$, this would force
\[
\frac{\omega^2-1}{\omega-1} = \frac{n}{m}\in\mathbb R.
\]
But, we know
\[
\frac{\omega^2-1}{\omega-1} = -\omega^2\notin\mathbb R.
\]
Hence, such a leaf has no non-zero period and is biholomorphic to $\mathbb C$. If exactly one of the
eigen-coordinates $u,v,w$ vanishes, the leaf lies in an invariant projective line; after deleting its two
singular endpoints it is biholomorphic to $\mathbb C^*$. Thus every regular leaf is biholomorphic
to $\mathbb C$ or $\mathbb C^*$, and in particular no leaf is elliptic. This proves condition~(3)
and the proposition.
\end{proof}

\subsection{The higher-degree case}

\begin{lem}\label{lem:no-elliptic-leaf}
For every $d \geq 2$, the foliation $\mathcal J_d$ has no leaf biholomorphic to an elliptic curve.
\end{lem}

\begin{proof}
The Jouanolou foliation has no invariant algebraic curve; see Jouanolou \cite{Jou79} and Lins Neto
\cite{Lin88}. If a leaf $L$ were biholomorphic to an elliptic curve, then $L$ would be compact. The
inclusion of the leaf into $\mathbb P^2$ would therefore be a proper holomorphic immersion, so its
image would be a compact one-dimensional analytic subset invariant under $\mathcal J_d$. By Chow's
theorem, it would be algebraic, a contradiction.
\end{proof}

\begin{proof}[Proof of Theorem~\ref{thm:jouanolou-type}]
The degree-one case is Proposition~\ref{prop:d1-jouanolou}. We therefore assume $d \geq 2$. For
$J_d = (y^d,z^d,x^d)$, the homogeneous vector field has zero divergence and is non-zero away from
the
origin. Choose $\varepsilon > 0$ and then $\eta > 0$ as in Sections~4--5, and let
$H = H_{\varepsilon,\eta}$, with $p = d+1$, and $N = H^{1/p}$.

\noindent We verify the three conditions of Definition~\ref{def:jouanolou-type} in Deroin's order.

\begin{itemize}
\item Proposition~\ref{prop:PS} proves more than hyperbolicity: every projective singularity is
hyperbolic and is a source for the associated real field $W$. Thus condition~(1) holds.

\item $N^2$ defines a smooth Hermitian metric on $\mathcal O(-1)$ and hence, using
$N\mathcal J_d\simeq\mathcal O(d+2)$, a smooth Hermitian metric on the normal bundle. The potential
satisfies
$P_B(H)$ by Proposition~\ref{prop:eta}, equivalently by Theorem~\ref{thm:gauge}(i).
Equation~\eqref{eq:inf-distance-core} identifies the logarithm of infinitesimal transverse distance
for this normal metric with a positive multiple of $f = -\log H$, up to an additive constant.
Therefore every regular critical point of the infinitesimal transverse distance is a strict local
maximum. This is condition~(2).

\item Lemma~\ref{lem:no-elliptic-leaf} gives condition~(3).
\end{itemize}
Hence $\mathcal J_d$ is of Jouanolou type for every $d \geq 2$. Together with
Proposition~\ref{prop:d1-jouanolou}, this proves \cite[Conjecture~4.6]{Der26}.

\end{proof}

\section{Consequences and remaining problems}

Theorem~\ref{thm:structural-stability}, together with the degree-one verification of Deroin, gives structural stability for the Jouanolou foliation in every degree, while
Theorem~\ref{thm:jouanolou-type} proves the corresponding all-degree Jouanolou-type statement.
For $d\geq 2$, the same construction gives an open set in $\operatorname{Fol}(\mathbb P^2,d)$ in which the Fatou set of a foliation is
a connected disc bundle, no leaf is dense, and all but countably many regular leaves are
biholomorphic to $\mathbb D$. Here, the main higher-degree difficulty is the
construction of an adapted transverse metric. The singularities of $\mathcal J_d$ are already
explicitly hyperbolic, and the absence of elliptic leaves follows from the classical absence of
invariant algebraic curves. The role of the exponent $p=d+1$ is to make the critical-point
equation compatible with the cyclic structure of the Jouanolou field and thereby produce a
homogeneous potential satisfying the required $P_B$ condition. This also suggests that adapted
homogeneous potentials may be useful for other projective foliations for which the Euclidean
potential does not detect the required transverse dynamics.

Deroin conjectures that the Julia set of $\mathcal J_d$ is transversely a
Cantor set of measure zero \cite[Conjecture~4.2]{Der26}. Theorem~\ref{thm:transverse-perfect}
establishes one necessary part of this picture, that is, at every regular Julia point, the intersection with
a sufficiently small holomorphic transversal has no isolated point. Thus transverse perfectness
is proved, but neither total disconnectedness nor zero transverse measure follows from the present
argument.

\medskip
\noindent\textbf{Problem 10.1 (Transverse structure of the Julia set).}
For $d\geq 2$, is the Julia set of $\mathcal J_d$ transversely a Cantor set of zero measure?

\medskip

The Fatou quotient gives another natural problem. For every $d\geq 2$,
Theorem~\ref{thm:dynamics-near-jd} produces a compact Riemann surface $B_d$ such that
\[
\mathcal F(\mathcal J_d)\longrightarrow B_d
\]
is a smooth locally trivial disc bundle and $g(B_d)=\frac{d(d+1)}2$. For $d=2$, Alvarez--Deroin identify $B_2$ with the Klein quartic \cite{AD25}. For $d>2$,
the argument above determines the genus but does not identify the conformal type of $B_d$.
The symmetries of the Jouanolou foliation act on the Fatou domain and therefore induce
automorphisms of its leaf-space quotient, so they may provide additional information about this
Riemann surface.

\medskip
\noindent\textbf{Problem 10.2 (The higher-degree Fatou quotient).}
For $d>2$, determine the conformal type of the Riemann surface $B_d$ and its automorphism group. 

\medskip

Structural stability and the Jouanolou-type property do not determine the minimal sets of the
foliation. In particular, the fact that no leaf is dense does not exclude the existence of an
exceptional minimal set. The results of Camacho--Lins Neto--Sad \cite{CLS88},
Bonatti--Langevin--Moussu \cite{BLM92}, and Camacho--de Figueiredo \cite{CF01} provide the
relevant background, but the structural-stability mechanism used here does not by itself settle
this question for the Jouanolou family.

\medskip
\noindent\textbf{Problem 10.3 (Exceptional minimal sets).}
Does $\mathcal J_d$ admit an exceptional minimal set for some $d\geq 6$? 

\medskip

It is also natural to ask how much of the structural-stability mechanism survives in higher
dimension. Several parts of the argument above use in an essential way that a one-dimensional
foliation on $\mathbb P^2$ has a complex normal line bundle. In that setting, the transverse
expansion is a scalar estimate on $N\mathcal F$, and the critical set $B$ is a real surface
transverse to the foliation, hence carries the transverse geometry of a Riemann surface. For a
one-dimensional foliation on $\mathbb P^n$, $n>2$, the normal bundle has complex rank $n-1$.
The natural homogeneous form
\[
\iota_V\iota_R\Omega
\]
is then an $(n-1)$-form and controls the determinant of the normal dynamics rather than the
behaviour of each normal direction separately. Thus transverse volume expansion is not, by
itself, a substitute for the uniform expansion of the whole normal bundle which is available in
rank one. Moreover, a smooth critical set transverse to the foliation would have real dimension
$2n-2$, so the Riemann-surface arguments used for the Fatou quotient and for several global
parts of the Alvarez--Deroin construction would require higher-dimensional replacements.

On the other hand, well-behaved local singularities are not scarce in higher dimension. In
a separate work \cite{GN26}, the author along with Vi\^{e}t-Anh Nguy\^{e}n proved that foliations with only hyperbolic linearizable singularities are dense in the corresponding parameter spaces $\operatorname{Fol}(\mathbb P^n,d)$, $n \ge 2$. This suggests that the main
difficulty in extending the present mechanism is not merely the local nature of the singularities,
but rather the construction and control of a higher-rank transverse dynamics.

\medskip
\noindent\textbf{Problem 10.4 (Structural stability in $\mathbb P^n$).}
Is there an analogue of the Alvarez--Deroin structural-stability mechanism for one-dimensional
holomorphic foliations on $\mathbb P^n$, $n>2$? 

\section*{Acknowledgements}
The author acknowledges support from the DST INSPIRE Faculty research grant (DST/INSP-IRE/04/2024/003868) and the Research Initiation Grant (IITJ/R\&D/IGRC/2025-26/53) from IIT Jodhpur.

\section*{Declarations}

\noindent\textbf{Competing Interests:} The author has no competing interests to declare that are
relevant to the content of this article.

\medskip

\noindent\textbf{Data Availability:} Data sharing is not applicable to this article.

\end{document}